\documentclass{amsart}
\usepackage{amssymb, amsmath, amsthm}
\numberwithin{equation}{section}
\usepackage{cite}                  
\usepackage[margin=1in]{geometry}
\usepackage{appendix}    
\usepackage[final]{graphicx}
\usepackage{enumitem}
\usepackage{subcaption} 
\usepackage{float}
\usepackage{subcaption} 
\usepackage{float}
\usepackage{mathrsfs}
\usepackage{array}\usepackage[bookmarks,bookmarksnumbered,bookmarksopen,breaklinks,linktocpage]{hyperref}
\usepackage{nomencl}
\newcommand{\ds}{\displaystyle}
\newcommand{\dd}{\mathrm{d}}

\newcommand{\K}{\mathcal{K}}
\newcommand{\C}{\mathcal{C}}

\newtheorem{Theorem}{Theorem}[section]
\newtheorem{Proposition}{Proposition}[section]
\newtheorem{Lemma}{Lemma}[section]
\newtheorem{example}{Test case}[section]
\theoremstyle{definition}           
\newtheorem{defn}{Definition}[section]
\theoremstyle{remark}
\newtheorem{Remark}{Remark}[section]
\begin{document}
	
	\title[From discrete to continuous condensing aggregation equation]{On the discrete to continuous condensing aggregation equation: A weak convergence approach}

	\author{Anupama Ghorai}
	\address{Department of Mathematics, National Institute of Technology Tiruchirappalli, Tamil Nadu - 620 015, India.}
	\email{anupamaghorai@gmail.com}
	\author{Jitraj Saha}
	\address{Department of Mathematics, National Institute of Technology Tiruchirappalli, Tamil Nadu - 620 015, India.}
	\email{jitraj@nitt.edu (Corresponding author)}
	
	\begin{abstract}
		 In this article, we study the passage of limits from  discrete to continuous condensing aggregation equation which comprises of Oort-Hulst-Safronov (OHS) equation  together with inverse aggregation process. We establish the relation between discrete and continuous condensing aggregation equations in its most generalized form, where  kinetic-kernels with respect to OHS and inverse aggregation equations are not always equal. Convergence criterion is proved under suitable \emph{a priori estimates} by approximating the continuous equation through a sequence of discrete equations, which subsequently converges towards the solution of the continuous equation by weak compactness principles. Existence of  solution to the discrete model and uniform bounds on different order moments over finite time under particular conditions on kinetic-kernels are investigated.
		We analyze long-time dynamics and   blowup of the solution leading to  mass-loss or gelation for specific kernels. Three numerical experiments show the accuracy and convergence of approximated solutions to the exact solution of the continuous equation when $\varepsilon$ approaches zero.
	\end{abstract}
	\maketitle

	\section{Introduction}\label{sec_1}
Particles suspended in a medium (typically a liquid or gas)  undergo random motion due to Brownian dynamics, leading them to move in various directions. This stochastic movement gives rise to different physical interactions such as aggregation, fragmentation, shattering, condensation etc. \cite{da1995existence,laurenccot2000class,banasiak2004conservativity,blair2007coagulation,banasiak2012analytic,banasiak2019analytic1,banasiak2019analytic2,giri2021weak} which significantly influence the evolution of particle size and their physical properties, thereby referring to as particulate systems.   Understanding and modelings these physical interactions in particulate systems have led to research and advancements  in various fields such as 
	crystallization, polymerization, astrophysics (formation of dust-clouds, planets, asteriods etc.), colloid chemistry, bio-science and precipitation  \cite{randolph2012theory,penlidis1986mathematical,MR1673235,wang2010aggregation, schwarzer2006predictive}.
	
In this study, we concentrate on aggregation process where the  particles  merge with each other to form larger agglomerates. These agglomerates are then responsible for alterations in characteristics of particles. 
	Thus, the particle size (or volume) increases and  total number of particles decreases over time evolution. In general, mass of the particles is expected to remain conserved for certain growth-rates. However, choices of kinetic-rates play a crucial role towards the conservation criterion, and hence formation of a large cluster  may lead to mass-loss (or gelation). 
	The pioneering work  on particulate modeling due to  two-particle aggregation   was introduced by renowned mathematician Smoluchowski \cite{smoluchowski1916drei} in 1916. Later,  
	Oort, Hulst and Safronov  \cite{oort1946gas,safronov1972evolution} proposed a new  equation, where evolution of particle growth depends on their sizes. Although binary particle encounter is taken into consideration,  mathematical formulation is  different from the classical Smoluchowski aggregation model.  
	From the name of the researchers involved, this aggregation model is known as Oort-Hulst-Safronov (OHS) equation,  defined by  
	\begin{align}\label{1_1}
		\partial_t{f}(t,x)=-\partial_x\left(f(t,x)U(t,x)\right)-W(t,x), \quad\text{for all}\quad (t,x)\in \mathbb{R_+}\times\mathbb{R_+},
	\end{align}
	where $\mathbb{R_+}:=(0,+\infty)$, 
	\begin{align}\label{1_2}
		U(t,x):=\int_{0}^{x} y\K(x,y)f(t,y)\dd y,\quad
		\text{and}\quad	W(t,x):=\int_{x}^{\infty} \K(x,y)f(t,x)f(t,y)\dd y.
	\end{align}
	The function $f(t,x)$  represents the number density of particles of size $x$ at time $t$, $U(t,x)$ denotes the growth-rate of formation of the particles of size $x$  from  smaller ones and $W(t,x)$ accounts the depletion of particles of size $x$ when they coalesce with larger particles \cite{MR1674432}.
	In  later  exploration, Davidson et al.  \cite{MR3640930} addressed a new, combined equation which is a coupling of OHS equations \eqref{1_1}-\eqref{1_2} together with an inverse aggregation process whose mechanism is contrary to  OHS  equations  \eqref{1_1}-\eqref{1_2}. Here, the evolution of particle growth  occurs through  the birth of particles of size $x$ from larger ones. The disappearance or annihilation of particles of  size $x$ happens  when they merge with smaller particles.   This  competing procedure is referred to as inverse aggregation process. Mathematically, inverse aggregation   reads as  the following  transport equation:
	\begin{align}\label{1_3}
		\partial_t{f}(t,x)=-\partial_x\left(f(t,x)V(t,x)\right)-Z(t,x), \quad\text{for all}\quad (t,x)\in\mathbb{R_+}\times\mathbb{R_+},
	\end{align}
	where
	\begin{align}\label{1_4}
		V(t,x):=\int_{x}^{\infty} y\C(x,y)f(t,y)\dd y,\quad
		\text{and}\quad  Z(t,x):=\int_{0}^{x} \C(x,y)f(t,x)f(t,y)\dd y.
	\end{align}
	The function $V(t,x)$ defines the growth-rate to form the particles of  size $x$ from larger ones and $Z(t,x)$ interprets similar to $W(t,x)$ but due to coalesce with smaller particles. The aggregation kernels $\K(x,y)$ and $\C(x,y)$ represent  the kinetic-rates at which   particles of size $x$  agglomerate with  particles of  size $y$  to form larger ones of size $(x+y)$. In general, these kernels are assumed to be nonnegative and symmetric with respect to the arguments $x$ and $y$. Note that  growth-rates $U(t,x)$ and $V(t,x)$
	depend on the density of  particles, but is independent on the size of  particles taken place in this process.
	Davidson et al. \cite{MR3640930} named OHS equations \eqref{1_1}-\eqref{1_2} together with  inverse aggregation equations \eqref{1_3}-\eqref{1_4} as the  \emph{condensing aggregation} (CA) equation which is written as
	\begin{align}\label{1_5}
		\partial_t{f}(t,x)&=\notag -\partial_x\bigg[f(t,x)\bigg(U(t,x)+V(t,x)\bigg)\bigg]-W(t,x)-Z(t,x)\notag\\
		&=:Q(f)(t,x),\quad\text{for all}\hspace{0.2cm}(t,x)\in\mathbb{R_+}\times\mathbb{R_+},
	\end{align}
	with the initial conditions,
	\begin{align}\label{1_6}
		f(0,x)=f^{\text{in}}(x)\geq0,\quad\text{for all}\hspace{0.2cm} x\in\mathbb{R_+}.
	\end{align}
	This equation \eqref{1_5} is known as the continuous condensing aggregation  (CCA) equation. 
	  In absence of terms  $W(t,x)$ and $Z(t,x)$, the equation \eqref{1_5} reduces to a continuity equation with density $f(t,x)$ and velocity $U(t,x)+V(t,x)$.

	It is worth  mentioning that if $\K=\C$,  the equation \eqref{1_5} becomes 
	\begin{align}\label{1_7}
		\partial_t{f}(t,x)	
		=-\partial_x\left[f(t,x)\int_{0}^{\infty} y\K(x,y)f(t,y)\dd y\right]-\int_{0}^{\infty} \K(x,y)f(t,x)f(t,y)\dd y.
	\end{align}
	This is a simplified mathematical equation provided that the integrals in  RHS of  the above equation exist.	 The physical importance of  equation \eqref{1_7} is yet to discuss. However, for certain choices of kernels exact solutions are available in \cite{MR3640930}.
	
	\subsection{State of the art and motivation}
	In the literature,	the  CA equation  has been described  in both its  discrete and  continuous forms  \cite{MR3640930}.
	Naturally particle population exists in a disperse framework. So, particle size can range from dust-ones (i.e. particle of nearly size zero)  to infinitely large clusters --
	though such events are rare. 
	To predict the dynamics  and impact of dust-size particle population in a variety of applications such as  aerosol science, chemical engineering and planetary formation  \cite{yu2017verification,di2011modeling}, researchers rely on the continuous equation. In contrast,  the discrete  equation \cite{MR1084278,MR4313094} explicitly depicts the evolution of the size (or volume) particle population categorized  into a well-defined quantized spectrum (i.e. positive integers).
	In discrete framework, a $i$-cluster consists of $i$-base particles of size $m_0$. Thus, setting $x=im_0$  and subsequently  passing $m_0\to 0$ one can  generate the analogue continuous equation  from  the discrete equation. This flexibility allows researchers to bridge the gap between discrete and continuous frameworks, for accurately modeling  a wide number of systems and phenomena. 
	In the discrete regime \cite{MR3640930},  
	the CA equations  \eqref{1_5}-\eqref{1_6}  is redefined by approximating $\partial_x$  by an upwind difference scheme and the integrals by Riemann sums. 
	Therefore, the discrete CA (DCA) equation is the time evolution of particle size distribution  $c_i(t)$ (with $c_0(t)=0$ for any $t\geq0$) of size $i$, given  by for all $t\geq0$ and  $i\geq1$,
	\begin{align}\label{1_8}
		\frac{\dd c_i(t)}{\dd t} = Q_i(c(t)),\quad \text{with initial data,}\quad c_i(0)=c^{\text{in}}_i\geq0,
	\end{align}
	where $c=\{c_i\}_{i\geq1}$,
	\begin{align}\label{1_9}
		Q_i(c(t)):=&\notag c_{i-1}(t)\sum_{j=1}^{i-1}j \K_{i-1,j}  c_{j}(t) - c_i(t) \sum_{j=1}^{i} j\K_{i,j}  c_j(t)-\sum_{j=i}^{\infty} \K_{i,j} c_i(t)c_j(t)\\
		&+c_{i-1}(t) \sum_{j=i-1}^{\infty} j\C_{i-1,j}  c_{j}(t) - c_i(t) \sum_{j=i}^{\infty} j\C_{i,j}  c_j(t)-\sum_{j=1}^{i} \C_{i,j} c_i(t)c_j(t).
	\end{align}
	The functions $\K_{i,j}$ and $\C_{i,j}$ 
	define the discrete aggregation and  inverse aggregation kernels respectively and are assumed to be nonnegative and symmetric with  the arguments $i$ and $j$. For a sequence of  sufficiently  decaying real numbers $\{\phi_i\}_{i\geq1}$, 
	the weak formulation (or moment equation) corresponding to the DCA equations  \eqref{1_8}-\eqref{1_9} is written as
	\begin{align}\label{1_10}
		\frac{\dd}{\dd t}\sum_{i=1}^{\infty}\phi_ic_i(t)=\sum_{i=1}^{\infty}\sum_{j=1}^{i}\left[j(\phi_{i+1}-\phi_i)-\phi_j\right]\K_{i,j}c_i(t)c_j(t)+\sum_{i=1}^{\infty}\sum_{j=i}^{\infty}\left[j(\phi_{i+1}-\phi_i)-\phi_j\right]\C_{i,j}c_i(t)c_j(t).
	\end{align}
	Mathematical studies such as existence, uniqueness, nonexistence of mass conserving solution and self-similar  solutions can be found  for OHS equations \eqref{1_1}-\eqref{1_2} in \cite{MR1674432,MR2000977,MR2338411, MR2142931,MR3238514,MR4313094,MR4421757,barik2022mass,dastrend}.
	However, due to the presence of nonlinear and nonlocal terms, the  CA equations \eqref{1_5}-\eqref{1_6}  is less explored. Moreover, the contrasting behavior of OHS  and inverse aggregation equations  makes the model more challenging to address.	In  literature,  Davidson et al. \cite{MR3640930}  documented    the  mathematical aspects such as self-similar solution, equilibrium  solution and exact solutions of CA equation  for certain choices of  kernels.
	In this context, our objective is to explore the passage of limits and establish the relation between discrete \eqref{1_8}-\eqref{1_9} and continuous forms \eqref{1_5}-\eqref{1_6} of the CA equations. 
	This comparative analysis demonstrates that for a fairly broad class of  continuous CA equations, there exists a  class of  approximated  discrete CA equations whose sequence  approaches to the solution of	the continuous CA equation when  limiting conditions are applied. Thus, this  convergence theory will provide concepts about their interconnected dynamics and  contribute to the particle growth governed by nonlinear and nonlocal interactions.  
	The theoretical framework for such convergence has been  developed in some related models. Notably, Bagland  el al. \cite{MR2158221} demonstrated  such exercise for 
	OHS   equations \eqref{1_1}-\eqref{1_2}.
	 However, when the OHS model is coupled with inverse aggregation equation—yielding  the condensing aggregation equation, it is not evident whether the convergence theory still reliable to the complete coupled model. In fact, no rigorous convergence results for the CA equation have been explored to date. This lack of theoretical  guarantees represents a significant gap in the literature and provides strong motivation for our present study, where our aim is to establish and validate a comprehensive convergence analysis for the CA equation. Furthermore, alongside the development of a convergence framework, we also aim to examine the existence of the solution to the discrete CA equations under some admissible kinetic-rates.

	In this article, we focus on identifying the conditions under which discrete CA equation can be linked to continuous CA equation, particularly in the context of growth conditions on kinetic-kernels $\K$ and $\C$. Our  convergence strategy is inspired from the works of Lauren{\c{c}}ot \cite{MR1938720} and Bagland \cite{MR2158221}.
	Furthermore,  we discuss \emph {a priori estimates} and prove  weak convergence of the sequence 
	of  approximated discrete equations to a  solution of  continuous CA equation under weak compactness principles.  
	We  also show  the existence of  solutions to the discrete CA equation for specific choices of kinetic-kernels. The  propagation of moments, long-time behavior and occurrence of gelation for some particular choices of kernels are  investigated thoroughly.
	Insights to the theoretical results are examined using numerical experiments. In numerical computations, all the possible choices of the kernels starting from simplified $\K=\C$ to $\K\neq\C$ and  reduction to OHS equation are discussed. An excellent agreement with the exact solution (wherever available) is observed as $\varepsilon\to 0$. Also the relative $L^1$ error for test cases (wherever exact solutions available) is performed  for  accuracy of the solutions.

	The  paper is organized as follows:  section \ref{sec_3} records main results including assumptions on kinetic-kernels, the construction of  a sequence of approximated discrete CA equations with appropriate discretization of associated functions. In section \ref{sec_4},  convergence of a sequence of the solutions of the discrete equations towards the  solution of  continuous CA equation is proved. We  also emphasize  on existence of solutions of  discrete CA equations with certain estimates in section \ref{sec_5}. In section \ref{sec_6},  propagation of moments is examined along with long-time behavior and occurrence of gelation. Section \ref{sec_7} carries out the numerical results.

\section{Main results}\label{sec_3}

\subsection{Continuous regime}
 
\subsubsection{Construction and assumption}\label{subsec_2}
For $1\le p< \infty$,  
$W^{1,p}(\Omega)$ is the Sobolev space defined by 
\begin{align*}
	W^{1,p}(\Omega):=\left\{u\in L^p(\Omega): \hspace{0.1cm}D^{q}u\in L^p(\Omega),\hspace{0.1cm} \text{for all}\hspace{0.2cm} |q|\le 1\right\},
\end{align*}
where $D^{q}u$ denotes the distributional derivatives of $u$. Also, let  $W^{1,\infty}_{\text{loc}}(\Omega)$ be  the space of functions $u$ such that both $u$ and its weak derivative $Du$ are  essentially bounded on every compact subset $V$ of $\Omega$,  that is, $u, Du\in L^{\infty}_{\text{loc}}(\Omega)$. 
Consider a weighted $L^1$ space defined by
${L}^1_{1}(\mathbb{R_+}):=L^1(\mathbb{R_+},(1+x)\dd x)$, and 
assume that 
\begin{align}\label{2_1}
	f^{\text{in}}\in L^1_{1}(\mathbb{R_+}),
	\quad\text{and}\quad f^{\text{in}}\geq0\quad\text{a.e.}
\end{align}
Throughout the paper we consider
the  following properties:\\ 
$(i)$\hspace{0.1cm} the kernels satisfy
\begin{align}\label{2_3}
	\K, \C\in W^{1,\infty}_{\text{loc}}([0,+\infty)^2), \quad\text{and}\quad
\end{align}
$(ii)$\hspace{0.1cm} for some constants $\alpha$, $ \beta\geq0$,
\begin{align}\label{2_4}
	\partial_x\K(x,y)\geq -\alpha,\quad\text{and}\quad\partial_x\C(x,y)\geq -\beta. 
\end{align}
\textit{Hypotheses in the continuous regime: For every $R\geq1$, $\K$ and $\C$ satisfy the following growth conditions respectively\\
	(CH1):	
	$\ds	\upsilon_R(y)=\sup_{x\in[0,R]}\frac{\K(x,y)}{y}\to 0\quad \text{as}\quad y\to +\infty,\quad\text{and}$\\
	(CH2): 
	$\ds \sup_{\substack{x\in[0,R]\\y\geq R}}\C(x,y)\le \mathcal{M}$, where  $\mathcal{M}\geq1$ is a constant.}
	
	Here, \textit{(CH1)} indicates that the kernel $\K$ grows strictly sub-linearly in the larger particle size $y$. This ensures tail control for the reaction terms. Likewise \textit{(CH1)}, \textit{(CH2)} allows that $\C$  stays uniformly bounded when particle's size $y$ is very large. This prevents blow-ups in the reaction nonlinear terms on truncated domains.
%

We are  in the stage to define the weak solution to the  CCA equations \eqref{1_5}-\eqref{1_6}.  
Let $C([0,T];\omega-L^1(\mathbb{R_+}))$  denote the space of weakly continuous functions from $[0,T]$ in $L^1(\mathbb{R_+})$. 
\begin{defn}[Weak solution]
	\textit{Let $T\in\mathbb{R_+}$ and $0\leq t\leq T$. Assume that $\K(x,y)$ and $\C(x,y)$ are nonnegative and symmetric satisfying the general conditions given in \eqref{2_3}-\eqref{2_4} and growth conditions (CH1) and (CH2) and  $f^{\textnormal{in}}\in L_1^1(\mathbb{R}_+)$. A function $f=f(t,x)$ is a weak solution to the CCA equations \eqref{1_5}-\eqref{1_6} 
		if 
		$$0\le f\in C\left([0,T];\omega-L^1(\mathbb{R_+})\right)\cap L^{\infty}\left(0,T;L^1_{1}(\mathbb{R_+})\right),$$
		and	for each $\phi\in\mathcal{D}(\mathbb{R_+})$, $f$ holds
		\begin{align}\label{2_5}
			\notag	\int_{0}^{\infty}\bigg(f(t,x)-f^{\textnormal{in}}(x)\bigg)\phi(x)\dd x
			=& \int_{0}^{t}\int_{0}^{\infty}\int_{0}^{x}\K(x,y)f(s,x)f(s,y)\left[y\partial_x\phi(x)-\phi(y)\right]\dd y\dd x\dd s\\	&+\int_{0}^{t}\int_{0}^{\infty}\int_{x}^{\infty}\C(x,y)f(s,x)f(s,y)\left[y\partial_x\phi(x)-\phi(y)\right]\dd y\dd x \dd s,
	\end{align}}
	where $\mathcal{D}(\mathbb{R_+})$ is the space of $C^{\infty}$-smooth functions with compact support on $\mathbb{R_+}$.
\end{defn}
Our approach regarding the convergence study involves constructing the step-wise approximations to the  functions $f$, $\K$, $\C$, $\phi$ and derivative of $\phi$. For this, we now specify them  in the following discretized structure similarly as in \cite{MR2158221}.
\subsubsection{ Approximations and reformulations}
Set $\varepsilon\in(0,1)$ and
$\ds 	\Lambda^{\varepsilon}_{i}:=[(i-1/2)\varepsilon,(i+1/2)\varepsilon)$, for  all $i\geq1$
and let $1_{\Lambda^{\varepsilon}_{i}}$ denote the characteristic function defined on ${\Lambda^{\varepsilon}_{i}}$.  Then, for $(t,x,y)\in\mathbb{R}^{3}_{+}$, define
\begin{align}\label{2_6}
		&f_{\varepsilon}(t,x)=\ds \sum_{i=1}^{\infty}c_i(t)1_{\Lambda^{\varepsilon}_{i}}(x),\\
		\K_{\varepsilon}(x,y)= \ds \sum_{i,j=1}^{\infty}&\frac{\K_{i,j}}{\varepsilon}1_{\Lambda^{\varepsilon}_{i}}(x)1_{\Lambda^{\varepsilon}_{j}}(y),\quad\text{and}\quad
		\C_{\varepsilon}(x,y)= \ds \sum_{i,j=1}^{\infty}\frac{\C_{i,j}}{\varepsilon}1_{\Lambda^{\varepsilon}_{i}}(x)1_{\Lambda^{\varepsilon}_{j}}(y).
\end{align}  
For every $x\in\mathbb{R_+}$ and $\phi\in\mathcal{D}(\mathbb{R_+})$,  define 
\begin{align}\label{2_8}
	\phi_{\varepsilon}(x)=\sum_{i=1}^{\infty}\phi^{\varepsilon}_i1_{\Lambda^{\varepsilon}_{i}}(x),
	\quad\text{with}\quad\phi^{\varepsilon}_i=\frac{1}{\varepsilon}\int_{\Lambda^{\varepsilon}_{i}}\phi(y)\dd y,
\end{align}  
and 
$g:\mathbb{R_+}\to\mathbb{R}$ as
\begin{align}\label{2_9}
	g(x)=\sum_{i=1}^{\infty}g_i1_{\Lambda^{\varepsilon}_{i}}(x),
	\quad\text{for all}\quad g_i\in\mathbb{R},
\end{align}  
with the discrete size derivative $D_{\varepsilon}(g)$ of $g$ as 
\begin{align}\label{2_10}
	D_{\varepsilon}(g)(x)=\frac{1}{\varepsilon}\sum_{i=1}^{\infty}(g_{i+1}-g_i)1_{\Lambda^{\varepsilon}_{i}}(x)
	,\quad\text{for all}\quad x\in\mathbb{R_+}.
\end{align}
Using these approximations \eqref{2_6}-\eqref{2_10}, 
equation \eqref{1_10} is  written as
\begin{align}\label{2_11}
	\frac{\dd}{\dd t}\int_{0}^{\infty}f_{\varepsilon}(t,x)\phi_{\varepsilon}(x)\dd x=&\notag \int_{0}^{\infty}\int_{0}^{r_{\varepsilon}(x)}\K_{\varepsilon}(x,y)f_{\varepsilon}(t,x)f_{\varepsilon}(t,y)\left[yD_{\varepsilon}(\phi_{\varepsilon})(x)-\phi_{\varepsilon}(y)\right]\dd y\dd x\\
	&+\int_{0}^{\infty}\int_{r_{\varepsilon}(x)}^{\infty}\C_{\varepsilon}(x,y)f_{\varepsilon}(t,x)f_{\varepsilon}(t,y)\left[yD_{\varepsilon}(\phi_{\varepsilon})(x)-\phi_{\varepsilon}(y)\right]\dd y\dd x,
\end{align}
where 
$\ds r_{\varepsilon}(x):=\left(\left[\frac{x}{\varepsilon}+\frac{1}{2}\right]\varepsilon+\frac{1}{2}\varepsilon\right) $,
and $[v]$  denotes the largest integer  $v$.\\
Suppose  $\left\{f_{\varepsilon}\right\}$ converges towards a function $f$ and if $\left\{\K_{\varepsilon}\right\}$, $\left\{\C_{\varepsilon}\right\}$, $\left\{\phi_{\varepsilon}\right\}$ and $\left\{ D_{\varepsilon}(\phi_{\varepsilon})\right\}$ converge towards $\K$, $\C$, $\phi$ and $\partial_x\phi$, respectively, then passing the limit $\varepsilon\to0$ in equation \eqref{2_11}, we get that for every $\phi\in\mathcal{D}(\mathbb{R_+})$, $f$ obeys
\begin{align}\label{2_12}
	\frac{\dd}{\dd t}\int_{0}^{\infty}f(t,x)\phi(x) \dd x=&\notag \int_{0}^{\infty}\int_{0}^{x}\K(x,y)f(t,x)f(t,y)[y\partial_x\phi(x)-\phi(y)]\dd y\dd x\\
	&+\int_{0}^{\infty}\int_{x}^{\infty}\C(x,y)f(t,x)f(t,y)[y\partial_x\phi(x)-\phi(y)]\dd y \dd x.
\end{align} 
Equation \eqref{2_12} is considered as the weak formulation of CCA equations \eqref{1_5}-\eqref{1_6}.
We next approximate the discretized form of all kinetic-rates and the density function  with initial data involved in the DCA equations \eqref{1_8}-\eqref{1_9}.  The continuous functions associated with  discreized  approximations are  also formulated.
\allowdisplaybreaks
\subsection{Discrete regime: approximations and weak reformulations}
For a fixed $\varepsilon\in(0,1)$, the discrete initial condition $c^{\text{in},\varepsilon}=\{c^{\text{in},\varepsilon}_{i}\}_{i\geq1}$ is defined  by
\begin{align}\label{2_13}
	c^{\text{in},\varepsilon}_{i}=\frac{1}{\varepsilon}\int_{\Lambda^{\varepsilon}_{i}}f^{\text{in}}(x)\dd x,\quad\text{for each}\quad i\geq1.
\end{align} 
We  next define discrete kernels,  
for all $i,j\geq1$
\begin{align}\label{2_14}
	\notag	\K^{\varepsilon}_{i,j}:=&\frac{1}{\varepsilon}\int_{\Lambda^{\varepsilon}_{i}\times\Lambda^{\varepsilon}_{j}}\K(x,y)
	\dd y\dd x\quad \text{or,}\quad\K^{\varepsilon}_{i,j}:=\varepsilon\K(\varepsilon i,\varepsilon j), \quad \text{and}\\
	\C^{\varepsilon}_{i,j}:=&\frac{1}{\varepsilon}\int_{\Lambda^{\varepsilon}_{i}\times\Lambda^{\varepsilon}_{j}}\C(x,y)\dd y\dd x \quad\text{or,}\quad\C^{\varepsilon}_{i,j}:=\varepsilon\C(\varepsilon i,\varepsilon j).
\end{align}
\noindent
\textit{The consequences of discrete approximations of kernels $\K$ and $\C$ for  hypotheses \textit{(CH1)-(CH2)} are given as follows: for all $i,j\geq1$ and $\varepsilon\in(0,1)$, \\
	(DH1): $\K^{\varepsilon}_{i,j}$ satisfies growth condition 
		$\ds \lim_{j\to+\infty}\frac{\K^{\varepsilon}_{i,j}}{j}=0,$\quad\text{and}\\
	(DH2): For some integer $m\geq1$, \; $\C^{\varepsilon}_{i,j}$ satisfies\;   
$\ds\sup_{j\geq m}\C^{\varepsilon}_{i,j}\le\varepsilon \mathcal{M}$, where  $\mathcal{M}\geq1$ is a constant.}\\
Therefore, for each $t\in[0,+\infty)$ and $i\geq1$,  the  DCA equations \eqref{1_8}-\eqref{1_9} is redefined  by
\begin{align}\label{2_15}
\frac{\dd c^{\varepsilon}_i(t)}{\dd t} = Q^{\varepsilon}_i(c^{\varepsilon}(t)),\quad\text{with  initial data,} \quad c^{\varepsilon}_i(0)=c^{\text{in},\varepsilon}_i,
\end{align}
where $c^{\varepsilon}=\{c^{\varepsilon}_i\}_{i\geq1}$,\quad and 
\begin{align}\label{2_16}
Q^{\varepsilon}_i(c^{\varepsilon}(t)):=&\notag c^{\varepsilon}_{i-1}(t) \sum_{j=1}^{i-1}j \K^{\varepsilon}_{i-1,j}  c^{\varepsilon}_{j}(t) - c^{\varepsilon}_i(t) \sum_{j=1}^{i} j\K^{\varepsilon}_{i,j} c^{\varepsilon}_j(t)-\sum_{j=i}^{\infty} \K^{\varepsilon}_{i,j} c^{\varepsilon}_i(t)\;c^{\varepsilon}_j(t)\\
&+c^{\varepsilon}_{i-1}(t) \sum_{j=i-1}^{\infty} j\C^{\varepsilon}_{i-1,j}  c^{\varepsilon}_{j}(t) - c^{\varepsilon}_i(t) \sum_{j=i}^{\infty} j\C^{\varepsilon}_{i,j} c^{\varepsilon}_j(t)-\sum_{j=1}^{i} \C^{\varepsilon}_{i,j} c^{\varepsilon}_i(t)c^{\varepsilon}_j(t).
\end{align}
Consequently, the weak formulation corresponding to the above DCA equations is calculated by 
\begin{align}\label{2_17}
\frac{\dd}{\dd t}\sum_{i=1}^{\infty}\phi^{\varepsilon}_ic^{\varepsilon}_i(t)=&\sum_{i=1}^{\infty}\sum_{j=1}^{i}\left[j(\phi^{\varepsilon}_{i+1}-\phi^{\varepsilon}_i)-\phi^{\varepsilon}_j\right]\K^{\varepsilon}_{i,j}c^{\varepsilon}_i(t)c^{\varepsilon}_j(t)\notag\\&+\sum_{i=1}^{\infty}\sum_{j=i}^{\infty}\left[j(\phi^{\varepsilon}_{i+1}-\phi^{\varepsilon}_i)-\phi^{\varepsilon}_j\right]\C^{\varepsilon}_{i,j}c^{\varepsilon}_i(t)c^{\varepsilon}_j(t),
\end{align}
where $\left\{\phi^{\varepsilon}_i\right\}_{i\geq1}$ is a sequence of positive real numbers.
For any particulate model,
the total mass of particles at time $t$ for  equations  \eqref{2_15}-\eqref{2_16}  in any closed system does not exceed the total mass of particles at initial time, that is,
\begin{align}\label{2_18}
\sum_{i=1}^{\infty} ic^{\varepsilon}_{i}(t)\le\sum_{i=1}^{\infty}ic^{\text{in},\varepsilon}_{i},\quad \text{for all}\quad t\geq0.
\end{align}
\subsubsection{Continuous formulations for the discrete approximations:}
For $(t,x,y)\in\mathbb{R}^{3}_{+}$, set
\begin{subequations}\label{2_19}
\begin{align}
	&f_{\varepsilon}(t,x)=\sum_{i=1}^{\infty}c^{\varepsilon}_i(t)1_{\Lambda^{\varepsilon}_{i}}(x)
	,\label{2_19a}\\ \K_{\varepsilon}(x,y)=\sum_{i,j=1}^{\infty}\frac{\K^{\varepsilon}_{i,j}}{\varepsilon}&1_{\Lambda^{\varepsilon}_{i}}(x)1_{\Lambda^{\varepsilon}_{j}}(y),\quad \text{and}\quad
	\C_{\varepsilon}(x,y)=\sum_{i,j=1}^{\infty}\frac{\C^{\varepsilon}_{i,j}}{\varepsilon}1_{\Lambda^{\varepsilon}_{i}}(x)1_{\Lambda^{\varepsilon}_{j}}(y)
	.\label{2_19b}
\end{align}
\end{subequations}
Here, $\left\{\K_{\varepsilon}\right\}$ and  $\left\{\C_{\varepsilon}\right\}$ converge towards $\K$ and $\C$  respectively when $\varepsilon\to 0$. 
Integrating the approximated function  $f_{\varepsilon}$ \eqref{2_19a} with respect to $x$ over $(0,+\infty)$, it follows that
\begin{align}\label{2_20}
\int_{0}^{\infty}f_{\varepsilon}(t,x)\dd x=\int_{0}^{\infty}\sum_{i=1}^{\infty}c^{\varepsilon}_i(t)1_{\Lambda^{\varepsilon}_{i}}(x)
\dd x=\varepsilon\sum_{i=1}^{\infty}c^{\varepsilon}_i(t).
\end{align}
Similarly, 
\begin{align}\label{2_21}
\int_{0}^{\infty}xf_{\varepsilon}(t,x)\dd x=\int_{0}^{\infty}\sum_{i=1}^{\infty}c^{\varepsilon}_i(t)x1_{\Lambda^{\varepsilon}_{i}}(x)
\dd x=\varepsilon^2\sum_{i=1}^{\infty}ic^{\varepsilon}_i(t).
\end{align}

\section{Weak convergence of solution}\label{sec_4}
To prove 
the  sequence $\left\{f_{\varepsilon}\right\}$  converges weakly to a function $f$ 
in $C([0,T];\omega-L^1(\mathbb{R_+}))$ for each $T\in\mathbb{R_+}$, 
we set \emph{a priori estimates}.  
The primary approach depends on
by taking uniform estimates with respect to  $\varepsilon$ for the approximated function $f_{\varepsilon}$ 
and finally passing the limit  $\varepsilon\to0$.
\subsection{A priori estimates}
Fix
\begin{align}\label{3_1}
||f^{\text{in}}||_{0,1}:=\int_{0}^{\infty}(1+x)f^{\text{in}}(x)\dd x.
\end{align}
For simplicity, we  write $c^{\varepsilon}_i$ instead of $c^{\varepsilon}_i(t)$.
\begin{Lemma}\label{lem_3_1}
For each $t\geq0$ and $\varepsilon\in(0,1)$, 
\begin{align}\label{3_2}
	\int_{0}^{\infty}xf_{\varepsilon}(t,x)\dd x\le2||f^{\textnormal{in}}||_{0,1},\quad \text{and}\quad\int_{0}^{\infty}f_{\varepsilon}(t,x)\dd x\le ||f^{\textnormal{in}}||_{0,1}.
\end{align}
\end{Lemma}
\begin{proof}
Recalling the definition \eqref{2_13}  and  summing over $i$, 
we get
\begin{align}\label{3_3}
	\varepsilon \sum_{i=1}^{\infty} c^{\text{in},\varepsilon}_{i} =\sum_{i=1}^{\infty} \int_{\Lambda^{\varepsilon}_{i}}f^{\text{in}}(x)\dd x
	\le\int_{\frac{\varepsilon}{2}}^{\infty}f^{\text{in}}(x)\dd x
	\le \int_{0}^{\infty}f^{\text{in}}(x)\dd x.
\end{align}
Similarly, 
\begin{align}\label{3_4}
	\varepsilon^2\sum_{i=1}^{\infty}i c^{\text{in},\varepsilon}_{i}\le 2\int_{0}^{\infty}xf^{\text{in}}(x)\dd x.
\end{align}
Applying  estimation \eqref{2_21} in the above inequality yields 

\begin{align*}
	\int_{0}^{\infty}xf_{\varepsilon}(t,x)\dd x=\varepsilon^2\sum_{i=1}^{\infty}ic^{\varepsilon}_i(t)\le \varepsilon^2\sum_{i=1}^{\infty}ic^{\text{in},\varepsilon}_i\le2\int_{0}^{\infty}xf^{\text{in}}(x)\dd x\le2||f^{\text{in}}||_{0,1}. 
\end{align*}
For $m\geq1$, substitute $\ds 
\phi^{\varepsilon}_i=
\begin{cases}
	1,& \mbox{if}\quad i\le m,\\
	0,&\mbox{if}\quad i > m,
\end{cases}$ 
in  weak formulation \eqref{2_17}.  Due to    nonnegativity of $\K^{\varepsilon}_{i,j}$, $\C^{\varepsilon}_{i,j}$ and  $c^{\varepsilon}_{i}$,  we obtain
\begin{align*}
	\frac{\dd}{\dd t}\left(\varepsilon\sum_{i=1}^{m}c^{\varepsilon}_i(t)\right)\le 0, \quad\text{ which implies}\quad\varepsilon\sum_{i=1}^{m}c^{\varepsilon}_i(t)\le \varepsilon\sum_{i=1}^{m}c^{\text{in},\varepsilon}_i\le \int_{0}^{\infty}f^{\text{in}}(x)\dd x\le ||f^{\text{in}}||_{0,1}.
\end{align*}
Therefore,  setting $m\to+\infty$ in above, we conclude 
\begin{align*}
	\int_{0}^{\infty}f_{\varepsilon}(t,x)\dd x\le\varepsilon\sum_{i=1}^{\infty}c^{\varepsilon}_i(t)\le ||f^{\text{in}}||_{0,1}.
\end{align*}
\end{proof}
\noindent To prove the upcoming Lemma, we will apply the following Lemma:
\begin{Lemma}\label{lem_2_1}(\hspace{-0.015cm}\cite{MR1892231})
			Let $\varphi\in C^2([0,+\infty))$ be a nonnegative convex function such that $\varphi(0)=0$, $\varphi'(0)=1$ and $\varphi'$ is concave. Then,
			$	u\varphi'(v)\le \varphi(u)+\varphi(v)$ and
			$\varphi(v)\leq	v\varphi'(v)\le 2\varphi(v)$, for all $u,v\geq0$.
		\end{Lemma}

\begin{Lemma}\label{lem_3_2}
Let $\varphi\in C^2([0,+\infty))$ be a nonnegative and convex function satisfying $\varphi(0)=0$, $\varphi'(0)=1$ and such that $\varphi'$ is concave.
Suppose that 
$\ds\Sigma_{\varphi}:=\int_{0}^{\infty}\varphi(f^{\textnormal{in}})(x)\dd x<+\infty$.
Then, for each $T\in\mathbb{R_+}$, there exists a constant $L(T)>0$ 
such that for any $t\in[0,T]$ and $\varepsilon\in(0,1)$, the following estimate holds: $$\ds	\int_{0}^{\infty}\varphi(f_{\varepsilon}(t,x))\dd x<L(T)\Sigma_{\varphi}.$$

\end{Lemma}
\begin{proof}
Let $T>0$ and choose $R>0$ and $m\in\mathbb{N}$ such that $R\in\Lambda^{\varepsilon}_{m}$. From the weak formulation  \eqref{2_17},  we obtain  
\begin{align*}
	\frac{\dd}{\dd t}\sum_{i=1}^{m}\varphi(c^{\varepsilon}_i)
	=\sum_{i=1}^{m}\varphi'(c^{\varepsilon}_i)\frac{\dd c^{\varepsilon}_i}{\dd t}
	=\sum_{i=1}^{m}\varphi'(c^{\varepsilon}_i)\bigg[&c^{\varepsilon}_{i-1} \sum_{j=1}^{i-1} j\K^{\varepsilon}_{i-1,j}  c^{\varepsilon}_{j} - c^{\varepsilon}_i \sum_{j=1}^{i} j\K^{\varepsilon}_{i,j} c^{\varepsilon}_j
	-\sum_{j=i}^{\infty} \K^{\varepsilon}_{i,j} c^{\varepsilon}_ic^{\varepsilon}_j\\	&+
	c^{\varepsilon}_{i-1} \sum_{j=i-1}^{\infty}j \C^{\varepsilon}_{i-1,j}  c^{\varepsilon}_{j} -
	c^{\varepsilon}_i \sum_{j=i}^{\infty} j\C^{\varepsilon}_{i,j} c^{\varepsilon}_j-\sum_{j=1}^{i} \C^{\varepsilon}_{i,j} c^{\varepsilon}_ic^{\varepsilon}_j\bigg].
\end{align*}
Using the nonnegativity of $\K^{\varepsilon}_{i,j}$, $\C^{\varepsilon}_{i,j}$, $c^{\varepsilon}_i$ and $\varphi'$,  
we  get
\begin{align}\label{3_6}
	\frac{\dd}{\dd t}\sum_{i=1}^{m}\varphi(c^{\varepsilon}_i)
	\le & \sum_{i=1}^{m-1}\sum_{j=1}^{i}j\K^{\varepsilon}_{i,j}c^{\varepsilon}_{i}   c^{\varepsilon}_{j}\varphi'(c^{\varepsilon}_{i+1}) - \sum_{i=1}^{m}\sum_{j=1}^{i} j\K^{\varepsilon}_{i,j} c^{\varepsilon}_i c^{\varepsilon}_j\varphi'(c^{\varepsilon}_i)\notag\\\notag
	& + \sum_{i=1}^{m-1}\sum_{j=i}^{\infty}j\C^{\varepsilon}_{i,j} c^{\varepsilon}_{i} c^{\varepsilon}_{j}\varphi'(c^{\varepsilon}_{i+1}) -\sum_{i=1}^{m}\sum_{j=i}^{\infty} j\C^{\varepsilon}_{i,j} c^{\varepsilon}_ic^{\varepsilon}_{j}\varphi'(c^{\varepsilon}_i)\\
	\le&\sum_{i=1}^{m-1}\sum_{j=1}^{i}j\K^{\varepsilon}_{i,j}c^{\varepsilon}_{i}  c^{\varepsilon}_{j}\left[\varphi'(c^{\varepsilon}_{i+1})-\varphi'(c^{\varepsilon}_i)\right] - \sum_{j=1}^{m} j\K^{\varepsilon}_{m,j} c^{\varepsilon}_{m} c^{\varepsilon}_j\varphi'(c^{\varepsilon}_{m})\notag\\
	& +\sum_{i=1}^{m-1}\sum_{j=i}^{\infty}j\C^{\varepsilon}_{i,j} c^{\varepsilon}_{i} c^{\varepsilon}_{j}\left[\varphi'(c^{\varepsilon}_{i+1})-\varphi'(c^{\varepsilon}_i)\right] -\sum_{j=m}^{\infty} j\C^{\varepsilon}_{m,j} c^{\varepsilon}_{m}c^{\varepsilon}_{j}\varphi'(c^{\varepsilon}_{m}).
\end{align}

Define,\;
$\ds \Phi(x):=x\varphi'(x)-\varphi(x),$  for all $x\in\mathbb{R_+}.$
Lemma \ref{lem_2_1} ensures that $\Phi$ satisfies the following properties:
\begin{align}\label{3_7}
	0\le\Phi(x)\le x\varphi'(x),\quad\text{ and}\quad\Phi(x)\le \varphi(x),\quad \text{for all}\quad x\in\mathbb{R_+}.
\end{align}  
Owing to convexity of  $\varphi$, we get   $(x-y)\varphi'(y)\leq\varphi(x)-\varphi(y)$, for all $x,y\in\mathbb{R_+}$ and hence,
\begin{align}\label{3_8}
	x(\varphi'(y)-\varphi'(x))\le\Phi(y)-\Phi(x).
\end{align}
Using Lemma \ref{lem_2_1} together with  above two relations  in inequality  \eqref{3_6} deduces 
\begin{align*}
	\frac{\dd}{\dd t}\sum_{i=1}^{m}\varphi(c^{\varepsilon}_i)\le&\sum_{i=1}^{m-1}\sum_{j=1}^{i}j\;\K^{\varepsilon}_{i,j}\;   c^{\varepsilon}_{j}\left[\Phi(c^{\varepsilon}_{i+1})-\Phi(c^{\varepsilon}_i)\right] - \sum_{j=1}^{m} j\K^{\varepsilon}_{m,j}\; c^{\varepsilon}_j\Phi(c^{\varepsilon}_{m})\\\notag
	&+\sum_{i=1}^{m-1}\sum_{j=i}^{\infty}j\C^{\varepsilon}_{i,j} c^{\varepsilon}_{j}\left[\Phi(c^{\varepsilon}_{i+1})-\Phi(c^{\varepsilon}_i)\right] -\sum_{j=m}^{\infty} j\C^{\varepsilon}_{m,j} c^{\varepsilon}_{j}\Phi(c^{\varepsilon}_{m}).
\end{align*}
Further simplifying,  we get
\begin{align}\label{3_9}
	\frac{\dd}{\dd t}\sum_{i=1}^{m}\varphi(c^{\varepsilon}_i)\le\sum_{i=2}^{m}\sum_{j=1}^{i-1}j\left[\K^{\varepsilon}_{i-1,j}-\K^{\varepsilon}_{i,j}\right ] c^{\varepsilon}_{j}\Phi(c^{\varepsilon}_i)+
	\sum_{i=2}^{m}\sum_{j=i-1}^{\infty}j\left[\C^{\varepsilon}_{i-1,j}-\C^{\varepsilon}_{i,j}\right] c^{\varepsilon}_{j}\Phi(c^{\varepsilon}_i).
\end{align}
Compiling the definitions of $\K^{\varepsilon}_{i,j}$ and $\C^{\varepsilon}_{i,j}$  (i.e. equations \eqref{2_14}) and  the assumption  \eqref{2_4}, we can write 
\begin{align}\label{3_10}
	\K^{\varepsilon}_{i-1,j}-\K^{\varepsilon}_{i,j}\le \alpha\varepsilon^2,\quad\text{and}\quad\C^{\varepsilon}_{i-1,j}-\C^{\varepsilon}_{i,j}\le \beta\varepsilon^2 \quad\text{for all}\quad i\geq2\quad \text{and}\quad j\geq1.
\end{align} 
Substituting  above inequalities in   \eqref{3_9} and then  using relation \eqref{3_7}, we obtain 
\begin{align*}
	\frac{\dd}{\dd t}\sum_{i=1}^{m}\varepsilon\varphi(c^{\varepsilon}_i)\le(\alpha+\beta)\sum_{i=1}^{m}\varepsilon \Phi(c^{\varepsilon}_i)
	\sum_{j=1}^{\infty}\varepsilon^2jc^{\varepsilon}_j\le&(\alpha+\beta)\sum_{i=1}^{m}\varepsilon \Phi(c^{\varepsilon}_i)
	\sum_{j=1}^{\infty}\varepsilon^2jc^{\text{in},\varepsilon}_j\\
	\le&2(\alpha+\beta)\sum_{i=1}^{m}\varepsilon \;\Phi(c^{\varepsilon}_i)
	\int_{0}^{\infty}xf^{\text{in}}(x)\dd x\\\le& 2(\alpha+\beta)||f^{\text{in}}||_{0,1}\sum_{i=1}^{m}\varepsilon \varphi(c^{\varepsilon}_i). 
\end{align*}
Gronwall's lemma leads us 
\begin{align}\label{3_11}
	\sum_{i=1}^{m}\varepsilon\varphi(c^{\varepsilon}_i)\le L(T)\sum_{i=1}^{m}\varepsilon \varphi(c^{\text{in},\varepsilon}_i),
\end{align}
where $ L(T):=2(\alpha+\beta)||f^{\text{in}}||_{0,1}$ is a constant, depending on $||f^{\text{in}}||_{0,1}$, $\alpha$ and $\beta$.
Thanks to Jensen inequality,  it  gives  
\begin{align}\label{3_12}
	\sum_{i=1}^{m}\varepsilon\varphi(c^{\varepsilon}_i)\le L(T)\sum_{i=1}^{m}\int_{\Lambda^{\varepsilon}_i} \varphi(f^{\text{in}}(x))\dd x\le L(T)\Sigma_{\varphi},\quad\text{for each }\quad t\in[0,T].
\end{align}
Since $(m+1/2)\varepsilon>R$, we estimate 
\begin{align}\label{3_13}
	\int_{0}^{R}\varphi(f_{\varepsilon}(t,x))\dd x\le \sum_{i=1}^{m}\varepsilon\varphi(c^{\varepsilon}_i(t))\le L(T)\Sigma_{\varphi}.
\end{align}
Passing $R\to+\infty$ in  above inequality \eqref{3_13},  we entail   
$\ds	\int_{0}^{\infty}\varphi(f_{\varepsilon}(t,x))\dd x\le L(T)\Sigma_{\varphi}.$	 
\end{proof}

\begin{Lemma}\label{lem_3_3} 
Let $C^1_{c}([0,+\infty))$ denote the space of continuously differentiable functions defined on $[0,+\infty)$ having compact support. Let $\varepsilon\in(0,1)$ and $\Phi\in C^1_{c}([0,+\infty))$. Then, for any $T\in\mathbb{R_+}$, 
$\ds\int_{0}^{\infty}f_{\varepsilon}(t,x)\Phi(x)\dd x$ is bounded in $W^{1,\infty}(0,T).$
\end{Lemma}

\begin{proof}
From the definition of $\Phi\in C^1_{c}([0,+\infty))$, we can take supp$(\Phi)\subset[0,R]$ for any  $R\in\mathbb{R}_+$.  For all $i\geq1$, we can define
\begin{align}\label{4_15}
	\Phi^{\varepsilon}_i:=\frac{1}{\varepsilon}\int_{\Lambda^{\varepsilon}_{i}}\Phi(y)\dd y.
\end{align}
Let $m$ be the integer such that $R\in \Lambda_m^{\varepsilon}$. 
Recalling weak formulation \eqref{2_17} gives
\begin{align}\label{3_15}
	\left|\frac{\dd}{\dd t}\int_{0}^{R}f_{\varepsilon}(t,x)\Phi(x)\dd x\right|=\varepsilon\;\left|\frac{\dd}{\dd t}\sum_{i=1}^{m}c^{\varepsilon}_i\Phi^{\varepsilon}_i\right|
	=&\varepsilon\bigg|
	\sum_{i=1}^{m}\sum_{j=1}^{i}j\K^{\varepsilon}_{i,j}c^{\varepsilon}_ic^{\varepsilon}_j(\Phi^{\varepsilon}_{i+1}-\Phi^{\varepsilon}_{i})-\sum_{i=1}^{m}\sum_{j=i}^{\infty}\K^{\varepsilon}_{i,j}c^{\varepsilon}_ic^{\varepsilon}_j\Phi^{\varepsilon}_{i}\notag\\\notag
	&\quad+\sum_{i=1}^{m}\sum_{j=i}^{\infty}j\C^{\varepsilon}_{i,j}c^{\varepsilon}_ic^{\varepsilon}_j(\Phi^{\varepsilon}_{i+1}-\Phi^{\varepsilon}_{i})-\sum_{i=1}^{m}\sum_{j=1}^{i}\K^{\varepsilon}_{i,j}c^{\varepsilon}_ic^{\varepsilon}_j\Phi^{\varepsilon}_{i}
	\bigg|\\\notag
	\le&\varepsilon^2
	\sum_{i=1}^{m}\sum_{j=1}^{i}j\K^{\varepsilon}_{i,j} c^{\varepsilon}_i c^{\varepsilon}_j\bigg|\frac{\Phi^{\varepsilon}_{i+1}-\Phi^{\varepsilon}_{i}}{\varepsilon}\bigg|+\underbrace{\varepsilon\sum_{i=1}^{m}\sum_{j=i}^{\infty}\K^{\varepsilon}_{i,j}c^{\varepsilon}_ic^{\varepsilon}_j|\Phi^{\varepsilon}_{i}|}_{\text{S}_1}\\
	&+\underbrace{\varepsilon^2\sum_{i=1}^{m}\sum_{j=i}^{\infty}j\C^{\varepsilon}_{i,j}c^{\varepsilon}_ic^{\varepsilon}_j\bigg|\frac{\Phi^{\varepsilon}_{i+1}-\Phi^{\varepsilon}_{i}}{\varepsilon}\bigg|}_{\text{S}_2}+\varepsilon\sum_{i=1}^{m}\sum_{j=1}^{i}\C^{\varepsilon}_{i,j}c^{\varepsilon}_ic^{\varepsilon}_j\left|\Phi^{\varepsilon}_{i}\right|.
\end{align}
Here, the first term and last term in the RHS of \eqref{3_15} are  sum of finite terms.
Observe that
\begin{align}\label{3_16}
	\left|\frac{\Phi^{\varepsilon}_{i+1}-\Phi^{\varepsilon}_{i}}{\varepsilon}\right|\le ||\Phi||_{W^{1,\infty}}.
\end{align}
Therefore,   
using the fact $\Phi\in C^1_{c}([0,+\infty))$, the term $\text{S}_1$ can be computed as
\begin{align}\label{3_17}
\text{S}_1	\leq \left(||\K||_{L^{\infty}((0,R+1)^2)}+2\left(\sup_{\substack{ i\le m \\ j\geq {m+1}}}\frac{\K^{\varepsilon}_{i,j}}{j\varepsilon^2}\right)\right)||f^{\text{in}}||_{0,1}^2||\Phi||_{W^{1,\infty}}.
\end{align}
Setting $L^{\infty}_R:=L^{\infty}{((0,R+1)^2)}$ and growth condition on $\C$,	the  term $\text{S}_2$ can be written as
\begin{align}\label{3_18}
	\notag	\text{S}_2=
	&\varepsilon^2\sum_{i=1}^{m}\left(\sum_{j=i}^{m}\C^{\varepsilon}_{i,j}jc^{\varepsilon}_ic^{\varepsilon}_j\bigg|\frac{\Phi^{\varepsilon}_{i+1}-\Phi^{\varepsilon}_{i}}{\varepsilon}\bigg|+\sum_{j=m+1}^{\infty}\C^{\varepsilon}_{i,j}jc^{\varepsilon}_ic^{\varepsilon}_j\bigg|\frac{\Phi^{\varepsilon}_{i+1}-\Phi^{\varepsilon}_{i}}{\varepsilon}\bigg|\right)\notag\\\leq& 2\left(||\C||_{L^{\infty}_R}+\mathcal{M}\right)||f^{\text{in}}||_{0,1}^2||\Phi||_{W^{1,\infty}}.
\end{align}
Using  estimations \eqref{3_16}-\eqref{3_18} in inequality  \eqref{3_15},
we  get
\begin{align}\label{3_19}
	\left|\frac{\dd}{\dd t}\int_{0}^{R}f_{\varepsilon}(t,x)\Phi(x)\dd x\right|
	\le \left[3\left(||\K||_{L^{\infty}_R}+||\C||_{L^{\infty}_R}\right)+2\left(
	\sup_{\substack{ i\le m\\ j\geq {m+1}}}
	\frac{\K^{\varepsilon}_{i,j}}{j\varepsilon^2}\right)
	+2\mathcal{M}\right]||f^{\text{in}}||_{0,1}^2||\Phi||_{W^{1,\infty}}.
\end{align}
By  growth condition $\textit{(CH1)}$, for every $R\geq1$ and $J>0$, there exists a nonnegative, bounded function $\hat{\upsilon}_{R}(J)$ such that 
\begin{align}\label{3_20}
	\sup_{\substack{ x\in[0,R] \\ y\geq J}}
	\frac{\K_{\varepsilon}(x,y)}{y}\le \hat{\upsilon}_{R}(J),\hspace{0.2cm}\text{and}\hspace{0.2cm}\lim_{J\to+\infty}\hat{\upsilon}_{R}(J)=0.
\end{align}
Hence,
\begin{align*}
	\left|\frac{\dd}{\dd t}\int_{0}^{R}f_{\varepsilon}(t,x)\Phi(x)\dd x\right|\leq  \left[3\left(||\K||_{L^{\infty}_R}+||\C||_{L^{\infty}_R}\right)+2\sup_{y\geq R}\hat{\upsilon}_{R+1}(y)
	+2\mathcal{M}\right]||f^{\text{in}}||_{0,1}^2||\Phi||_{W^{1,\infty}}=\mathcal{L}_R(\Phi).
\end{align*}
Passing to  limit $R\to+\infty$ in above inequality   concludes 
\begin{align*}
	\left|\frac{\dd}{\dd t}\int_{0}^{\infty}f_{\varepsilon}(t,x)\Phi(x)\dd x\right|<+\infty.
\end{align*}
Therefore,	this completes the proof of lemma.
\end{proof}

\subsection{Continuous regime: convergence and weak solution}

\begin{Lemma}\label{lem_3_4}
For each $\varepsilon\in(0,1)$ and $T\in \mathbb{R_+}$, 
there exists a nonnegative function $f$ and a subsequence of $\left\{f_{\varepsilon}\right\}$ such that 
\begin{align}\label{3_21}
	f\in L^{\infty}(0,T;L^1_{1}(\mathbb{R_+})), \quad \text{and}\quad f_{\varepsilon}\to f\quad\text{in}\quad C([0,T];\omega-L^1(\mathbb{R_+})).
\end{align}
\end{Lemma}
\begin{proof}
Let $T>0$.  Our aim is to prove 
that $\left\{f_{\varepsilon}\right\}$ is a relatively sequentially compact in $C([0,T];\omega-L^1(\mathbb{R_+}))$. We will prove it in the following two steps:\\
$(i)$ \textit{Weak compactness}: We have $f^{\text{in}}\in L^1(\mathbb{R_+})$. By de la Vall\`{e}e Poussin theorem (\hspace{-0.015cm}\cite{MR2341508}), there exists  
a function $\varphi$ satisfying  the assumptions of Lemma \ref{lem_3_2}   such that 
	$\ds	\lim_{r\to+\infty}\frac{\varphi(r)}{r}=+\infty,\quad\text{and}\quad\int_{0}^{\infty}\varphi(f^{\text{in}})(x)\dd x<+\infty.$
From Lemma \ref{lem_3_1}  and Lemma \ref{lem_3_2}, we obtain  
\begin{align}\label{3_23}
	\sup_{\substack{t \in[0,T]\\\varepsilon \in (0,1)}}
	\left\{\int_{0}^{\infty}(1+x)f_{\varepsilon}(t,x)\dd x+\int_{0}^{\infty}\varphi(f_{\varepsilon}(t,x))\dd x\right\}<+\infty.
\end{align}
By Dunford-Pettis theorem (\hspace{-0.015cm}\cite{MR2341508}),  there exists a  relatively weakly compact subset $\mathcal{E}$ of $ L^1(\mathbb{R_+})$ such that $f_{\varepsilon}\in\mathcal{E}$.\\ 
$(ii)$ \textit{Weak equicontinuity of $f_{\varepsilon}$}: Consider $\phi\in L^{\infty}(\mathbb{R_+})$. There exists a sequence of functions $\left\{\phi_k\right\}$ in $C^1_{c}(\mathbb{R_+})$ such that 
\begin{align}\label{3_24}
	\phi_k\to \phi \quad\text{a.e. in}\quad\mathbb{R_+},\quad\text{and}\quad
	||\phi_k||_{L^{\infty}}\le ||\phi||_{L^{\infty}}.
\end{align}
For setting $\sigma\in(0,1)$, the relation \eqref{3_23} ensures the existence of some real $\delta_1(\sigma)>0$ such that, for any measurable subset $X$ of $\mathbb{R_+}$ with meas$(X)\le \delta_1(\sigma)$, 
\begin{align}\label{3_25}
	\sup_{\substack{ t \in[0,T]\\\varepsilon \in (0,1)}}
	\int_{X}f_{\varepsilon}(t,x)\dd x\le  \sigma.
\end{align}
By  Egorov theorem  (\hspace{-0.015cm}\cite{MR2129625})
and estimation  \eqref{3_24}, there exists a measurable subset $X_{\sigma}$ of $[0,1/\sigma]$ such that 
\begin{align}\label{3_26}
	\text{ meas}(X_{\sigma})\le \delta_1(\sigma),\quad\text{and}\quad\lim_{k\to+\infty}\sup_{[0,1/\sigma]\setminus X_{\sigma}}|\phi_k-\phi|=0.
\end{align}
Therefore, for all $t\in(0,T)$, $h\in(-t,T-t)$ and $R\in[0,1/\sigma]$, we calculate 
\begin{align}\label{3_27}
	\bigg|\int_{0}^{\infty}\bigg(f_{\varepsilon}(t+h,x)-f_{\varepsilon}(t,x)\bigg )\phi(x) \dd x \bigg|\le& \bigg|\int_{0}^{R}\bigg(f_{\varepsilon}(t+h,x)-f_{\varepsilon}(t,x)\bigg)\;\phi_k(x)\dd x\bigg|\notag\\
	&+\bigg|\int_{0}^{R}\bigg(f_{\varepsilon}(t+h,x)-f_{\varepsilon}(t,x)\bigg)\bigg(\phi(x)-\phi_k(x)\bigg)\dd x\bigg|\notag\\
	&+\bigg|\int_{R}^{\infty}\bigg(f_{\varepsilon}(t+h,x)-f_{\varepsilon}(t,x)\bigg)\phi(x)\dd x\bigg|. 
\end{align}
Further rearrangements and estimates give
\begin{align}\label{3_28}
	\bigg|\int_{0}^{\infty}\bigg(f_{\varepsilon}(t+h,x)-f_{\varepsilon}(t,x)\bigg)\phi(x)\dd x \bigg|\notag\le& \left|\int_{t}^{t+h}\frac{\dd}{\dd s}\left(\int_{0}^{R}f_{\varepsilon}(s,x)\;\phi_k(x)\dd x\right)\dd s\right|\\\notag
	&+\bigg|\int_{[0,R]\setminus X_{\sigma}}\bigg(f_{\varepsilon}(t+h,x)-f_{\varepsilon}(t,x)\bigg)\bigg(\phi(x)-\phi_k(x)\bigg)\dd x\bigg|\\\notag
	&+\bigg|\int_{X_{\sigma}}\bigg(f_{\varepsilon}(t+h,x)-f_{\varepsilon}(t,x)\bigg)\bigg(\phi(x)-\phi_k(x)\bigg)\dd x\bigg|\\
	&+\bigg|\int_{R}^{\infty}\bigg(f_{\varepsilon}(t+h,x)-f_{\varepsilon}(t,x)\bigg)\phi(x)\dd x\bigg|.
\end{align} 
Recall the definitions of $\delta_1(\sigma)$, $X_{\sigma}$ and $\phi_k$. Therefore, by Lemma  \ref{lem_3_1} and Lemma \ref{lem_3_3} in the above inequality  yields 
\begin{align}\label{3_29}
	\bigg|\int_{0}^{\infty}\left[f_{\varepsilon}(t+h,x)-f_{\varepsilon}(t,x)\right]\phi(x)\dd x \bigg|\le |h|\mathcal{L}_R(\phi_k)+2||f^{\text{in}}||_{0,1}\sup_{[0,R]\setminus X_{\sigma}}|\phi_k-\phi|+2\bigg(2\sigma+\frac{||f^{\text{in}}||_{0,1}}{R}\bigg)||\phi||_{L^{\infty}},
\end{align}
where, $\ds \mathcal{L}_R(\phi_k)=\left[3\left(||\K||_{L^{\infty}_R}+||\C||_{L^{\infty}_R}\right)+2\sup_{y\geq R}\hat{\upsilon}_{R+1}(y)
+2\mathcal{M}\right]||f^{\text{in}}||_{0,1}^2||\phi_k||_{L^{\infty}}$ (derived in Lemma \ref{lem_3_3}).
Setting the limit $h\to0$ in \eqref{3_29}, we get
\begin{align}\label{3_30}
	\limsup_{h\to0}\sup_{\substack{t \in[0,T]\\ \varepsilon \in (0,1)}}
	\bigg|\int_{0}^{\infty}\left[f_{\varepsilon}(t+h,x)-f_{\varepsilon}(t,x)\right]\phi(x)\dd x \bigg|\le 2||f^{\text{in}}||_{0,1}\sup_{[0,R]\setminus X_{\sigma}}|\phi_k-\phi| +2\bigg(2\sigma+\frac{||f^{\text{in}}||_{0,1}}{R}\bigg)||\phi||_{L^{\infty}}.
\end{align}
Passing the  successive limits  $\sigma\to0$ and $k, R\to+\infty$ in  above inequality \eqref{3_30}, we obtain weakly equicontinuity of the family $\left\{f_{\varepsilon}:\varepsilon\in(0,1)\right\}$. Therefore, by $(i)$ and $(ii)$  $\left\{f_{\varepsilon}\right\}$ is relatively sequentially compact in $C([0,T];\omega-L^1(\mathbb{R_+}))$.
\end{proof}
It is remaining  to check that function $f$, defined by Lemma \ref{lem_3_4} is a weak solution to the  CCA equations \eqref{1_5}-\eqref{1_6}. 
Assume $\phi\in\mathcal{D}(\mathbb{R_+})$ and  $\phi_{\varepsilon}$  defined by \eqref{2_8}. For all $t\in \mathbb{R}_+$, we can easily examine  from CCA equations \eqref{1_5}-\eqref{1_6} that $f_{\varepsilon}$  holds
\begin{align}\label{3_31}
\notag\int_{0}^{\infty}\bigg(f_{\varepsilon}(t,x)-f_{\varepsilon}(0,x)\bigg)\phi_{\varepsilon}(x)\dd x=& \int_{0}^{t}\int_{0}^{\infty}\int_{0}^{r_{\varepsilon}(x)}\K_{\varepsilon}(x,y)\;f_{\varepsilon}(s,x)f_{\varepsilon}(s,y)\left[yD_{\varepsilon}(\phi_{\varepsilon})(x)-\phi_{\varepsilon}(y)\right]\dd y\dd x\dd s\\
&\hspace{0.1cm}+\int_{0}^{t}\int_{0}^{\infty}\int_{r_{\varepsilon}(x)}^{\infty}\C_{\varepsilon}(x,y)f_{\varepsilon}(s,x)f_{\varepsilon}(s,y)\left[yD_{\varepsilon}(\phi_{\varepsilon})(x)-\phi_{\varepsilon}(y)\right]\dd y\dd x\dd s.
\end{align}
For passing the  limit  $\varepsilon\to0$ in equation  \eqref{3_31}, the convergence results of $\phi_{\varepsilon}$,  $D_{\varepsilon}(\phi_{\varepsilon})$, $\K_{\varepsilon}$ and $\C_{\varepsilon}$ are required which are now established by the following lemma.
\begin{Lemma}\label{lem_3_5}
For each $R>0$ and $\varepsilon\in(0,1)$, the sequences $\left\{\phi_{\varepsilon}\right\}$, $\left\{\K_{\varepsilon}\right\}$ and $\left\{C_{\varepsilon}\right\}$   defined by \eqref{2_8} and \eqref{2_19b} respectively hold the following properties:\\
(i)	$||\phi_{\varepsilon}||_{L^{\infty}}\le ||\phi||_{L^{\infty}};\;||D_{\varepsilon}(\phi_{\varepsilon})||_{L^{\infty}}\le ||\phi||_{W^{1,\infty}};\;\phi_{\varepsilon}\to \phi \hspace{0.2cm}\text{and}\hspace{0.2cm}\;D_{\varepsilon}(\phi_{\varepsilon})\to \partial_x\phi \hspace{0.2cm}\text{strongly in}\hspace{0.2cm} L^{\infty}(\mathbb{R_+}),$\\
(ii) $||\K_{\varepsilon}||_{L^{\infty}((0,R)^2)}\le ||\K||_{L^{\infty}((0,R+1)^2)};\;||\C_{\varepsilon}||_{L^{\infty}((0,R)^2)}\le ||\C||_{L^{\infty}((0,R+1)^2)}$ and $\K_{\varepsilon}\to \K;
\;\C_{\varepsilon}\to \C $ a.e. on $\mathbb{R}^{2}_{+},$\\
(iii) $r_{\varepsilon}(x)\to x \hspace{0.2cm}\text{for a.e.} \hspace{0.2cm} x\in\mathbb{R_+}.
$
	\end{Lemma}
	
	\begin{proof}

	The proofs of the boundedness and convergence of $\phi_{\varepsilon}$, $D_{\varepsilon}$ and $\K_{\varepsilon}$ are referred to Bagland \cite{MR2158221}.
Now, the remaining part is to prove the boundedness and convergence of $\C_{\varepsilon}$. In order to do this, let $x,y\in\mathbb{R}_+$.  
By definition of $\C_{\varepsilon}$ \eqref{2_19b},
  we obtain 
\begin{align*}
	\left|\C_{\varepsilon}(x,y)\right|= \left|\frac{1}{\varepsilon^2} \sum_{i,j=1}^{\infty}1_{\Lambda^{\varepsilon}_{i}}(x)1_{\Lambda^{\varepsilon}_{j}}(y)\int_{\Lambda^{\varepsilon}_{i}\times\Lambda^{\varepsilon}_{j}}\C(u,v)\dd v \dd u\right|
	\le ||\C||_{{L^{\infty}((0,R+1)^2)}},
\end{align*}
 which implies the boundedness of $\C_{\varepsilon}$. For $x,y\in\mathbb{R}_+$, we estimate
\begin{align*}
	\left|\left|\C_{\varepsilon}(x,y)-\C(x,y)\right|\right|&=\left|\frac{1}{\varepsilon^2} \sum_{i,j=1}^{\infty}1_{\Lambda^{\varepsilon}_{i}}(x)1_{\Lambda^{\varepsilon}_{j}}(y)
	\int_{\Lambda^{\varepsilon}_{i}\times\Lambda^{\varepsilon}_{j}}\left(\C(u,v)-\C(x,y)\right)
	\dd v\dd u \right
	|
	\le\sup\limits_{\substack{|u-x|\le \varepsilon\\ |v-y|\le \varepsilon} }\left|\C(u,v)-\C(x,y)\right|,
\end{align*}
which gives $\C_{\varepsilon}\to\C$ a.e on $\mathbb{R}^{2}_{+}$ as $\varepsilon\to0$.\\
$(iii)$	Now, $\ds |r_{\varepsilon}(x)-x|
=\bigg|\bigg[\frac{x}{\varepsilon}+\frac{1}{2}\bigg]\varepsilon+\frac{1}{2}\varepsilon-\frac{x}{\varepsilon}\varepsilon\bigg|
=\varepsilon\bigg|1-\bigg(\bigg(\frac{x}{\varepsilon}+\frac{1}{2}\bigg)-\bigg[\frac{x}{\varepsilon}+\frac{1}{2}\bigg]\bigg)\bigg|
\le \varepsilon.$
This completes the proof.
\end{proof}

\begin{Lemma}\label{lem_3_6}
For every $T\in\mathbb{R_+}$ and $R>0$, we have
$\ds	f_{\varepsilon}(t,x)f_{\varepsilon}(t,y)\rightarrow f(t,x)f(t,y)$ in $\C([0,T];\omega-L^1((0,R)^2)).$

\end{Lemma}
\begin{proof} 
This proof follows directly from Lemma \ref{lem_3_4}. 
	\end{proof}
	\begin{Lemma}\label{lem_2_3}(\hspace{-0.015cm}\cite{MR1892231})
		Let $Y$ be an open bounded subset of $\mathbb{R}^p,\; p\geq1$ and consider two sequences $\left\{v_n\right\}$ in $L^1(Y)$ and $\left\{w_n\right\}$ in $L^{\infty}(Y)$ and a function $w\in L^{\infty}(Y)$ such that 
		$\left\{v_n\right\}\rightharpoonup v$ in $L^1(Y)$,\quad 
		$\ds||w_n||_{L^{\infty}}\le C$ and $w_n\to w$ a.e. in $Y$ 
		for some $C>0$. Then 
		\begin{align*}
				\lim_{n\to+\infty}||v_n\;(w_n-w)||_{L^1}=0,\quad\text{and}\quad v_n\;w_n\rightharpoonup vw\quad\text{in}\quad L^1(Y).
			\end{align*}
	\end{Lemma}
	
\begin{Theorem}\label{theo_3_1}
Assume that $\K$ and $\C$ are nonnegative and symmetric satisfying growth conditions (CH1) and (CH2) and the initial condition $f^{\textnormal{in}}$ satisfies \eqref{2_1}.
Let $f_{\varepsilon}$ be the function defined by \eqref{2_19a}. Then, there exists 
a subsequence  $\left\{f_{\varepsilon_n}\right\}_{n\geq1}$ of $\left\{f_{\varepsilon}\right\}$ such that
$$f_{\varepsilon_n}\to f \in   C([0,T];\omega-L^1(\mathbb{R_+})),\hspace{0.1cm}\text{ for every }\hspace{0.2cm} T\in\mathbb{R_+},$$ where  $f$ is the weak solution to the  CCA equations  \eqref{1_5}-\eqref{1_6}.   
\end{Theorem}
\begin{proof}
Let $T>0$. Consider $\phi\in\mathcal{D}(\mathbb{R_+})$ with supp$(\phi)\subset[0,I-2]$ for some $I>2$  
and $R>I$. Thus,  using convergence of $r_{\varepsilon}$, we get  $1_{[0,r_{\varepsilon}(x)]}\to 1_{[0,x]}$  and $1_{[r_{\varepsilon}(x), R]}\to 1_{[x, R]}$, for a.e. $x\in\mathbb{R_+}$. Therefore, Lemma \ref{lem_2_3} and Lemma \ref{lem_3_5} together with  previous estimates deduce  the following results:
\begin{align*}
\K_{\varepsilon}(x,y)\left[yD_{\varepsilon}\phi_{\varepsilon}(x)-\phi_{\varepsilon}(y)\right]1_{\left[0,r_{\varepsilon}(x)\right]}(y)  \longrightarrow\K(x,y)\left[y\partial_x\phi(x)-\phi(y)\right]1_{[0,\;x]}(y),\hspace{0.1cm}\text{and}\hspace{0.1cm}
\end{align*}
\begin{align*}
\C_{\varepsilon}(x,y)\left[yD_{\varepsilon}(\phi_{\varepsilon})(x)-\phi_{\varepsilon}(y)\right]1_{[r_{\varepsilon}(x),R]}(y) \longrightarrow\C(x,y)\left[y\partial_x\phi(x)-\phi(y)\right]1_{[x,R]}(y),
\end{align*}
a.e. in $(0,R)^2$ where $\phi_{\varepsilon}$ is  defined in \eqref{2_8}. 
Lemma \ref{lem_3_5}  recollects  bounds of $\phi_{\varepsilon}$, $D_{\varepsilon}(\phi_{\varepsilon})$,  $\K_{\varepsilon}$ and $\C_{\varepsilon}$ and then by Lemma \ref{lem_3_6},  setting  $v_{\varepsilon}(t,x,y):=f_{\varepsilon}(t,x)f_{\varepsilon}(t,y)$ and $w_{\varepsilon}(x,y):=\K_{\varepsilon}(x,y)\left[yD_{\varepsilon}(\phi_{\varepsilon})(x)-\phi_{\varepsilon}(y)\right]1_{[0,r_{\varepsilon}(x)]}(y)$, we get 
\begin{align}\label{3_32}
\int_{0}^{t}\int_{0}^{R}&\int_{0}^{R}\K_{\varepsilon}(x,y)f_{\varepsilon}(s,x)f_{\varepsilon}(s,y)\left[yD_{\varepsilon}(\phi_{\varepsilon})(x)-\phi_{\varepsilon}(y)\right]1_{[0,r_{\varepsilon}(x)]}(y)\dd y\dd x\dd s\notag\\
&\xrightarrow{\varepsilon\to 0}\int_{0}^{t}\int_{0}^{R}\int_{0}^{R}\K(x,y)f(s,x)f(s,y)\left[y\partial_x\phi(x)-\phi(y)\right]1_{[0,x]}(y)\dd y\dd x\dd s.
\end{align}
Since, supp$(\phi)\subset[0,R-2]$, therefore 
\begin{align*}
&\iint_{\mathbb{R}^{2}_{+}\setminus[0,R]^2}\K_{\varepsilon}(x,y)f_{\varepsilon}(s,x)f_{\varepsilon}(s,y)yD_{\varepsilon}(\phi_{\varepsilon})(x)1_{[0,r_{\varepsilon}(x)]}(y)\dd y\dd x=0,\quad\text{and}\\
&\iint_{\mathbb{R}^{2}_{+}\setminus[0,R]^2}\K(x,y)f(s,x)f(s,y)y\partial_x\phi(x)1_{[0,x]}(y)\dd y\dd x=0.
\end{align*}
Now, for $I<R$  we obtain
\begin{align*}
\bigg|\iint_{\mathbb{R}^{2}_{+}\setminus[0,R]^2}&	\K_{\varepsilon}(x,y)f_{\varepsilon}(s,x)f_{\varepsilon}(s,y)\phi_{\varepsilon}(y)1_{[0,r_{\varepsilon}(x)]}(y)\dd y\dd x\bigg|\\
&\le2\bigg|\int_{R}^{\infty}\int_{0}^{I}\K_{\varepsilon}(x,y)f_{\varepsilon}(s,x)f_{\varepsilon}(s,y)\phi_{\varepsilon}(y)\dd y\dd x\bigg|\\
&\le 2|\phi||_{L^{\infty}}\sup_{x\geq R}\hat{v}_I(x)\int_{0}^{\infty}f_{\varepsilon}(s,x)\dd x\int_{0}^{\infty}yf_{\varepsilon}(s,y)\dd y\\
&\le4 ||f^{\text{in}}||_{0,1}^2||\phi||_{L^{\infty}}\sup_{x\geq R}\hat{v}_I(x)
\end{align*}
and similarly, 
\begin{align*}
\bigg|\iint_{\mathbb{R}^{2}_{+}\setminus[0,R]^2}	\K(x,y)f(s,x)f(s,y)\phi(y)1_{[0,x]}(y)\dd y\dd x\bigg|
\le4 ||f^{\text{in}}||_{0,1}^2||\phi||_{L^{\infty}}\sup_{x\geq R}{v}_I(x).
\end{align*}
Hence, using growth condition both tends to $0$ as $R\to+\infty$ uniformly with respect to $\varepsilon$.
Using similar argument as in relation \eqref{3_32} for the second term in equation  \eqref{3_31}, we get
\begin{align}
\int_{0}^{t}\int_{0}^{R}\int_{0}^{R}&\C_{\varepsilon}(x,y)f_{\varepsilon}(s,x)f_{\varepsilon}(s,y)[yD_{\varepsilon}(\phi_{\varepsilon})(x)-\phi_{\varepsilon}(y)]1_{[r_{\varepsilon}(x),R]}(y)\dd y\dd x\dd s\notag\\
\xrightarrow{\varepsilon\to0}&\int_{0}^{t}\int_{0}^{R}\int_{0}^{R}\C(x,y)f(s,x)f(s,y)\left[y\partial_x\phi(x)-\phi(y)\right]1_{[x,R]}(y)\dd y\dd x\dd s.
\end{align} 
We have supp$(\phi)\subset[0,R-2]$, therefore
\begin{align*}
\iint_{\mathbb{R}^{2}_{+}\setminus[0,R]^2}\C_{\varepsilon}(x,y)f_{\varepsilon}(s,x)f_{\varepsilon}(s,y)yD_{\varepsilon}(\phi_{\varepsilon})(x)1_{[r_{\varepsilon}(x),R]}(y)\dd y\dd x=0,\\
\iint_{\mathbb{R}^{2}_{+}\setminus[0,R]^2}\C(x,y)f(s,x)f(s,y)y\partial_x\phi(x)1_{[x,R]}(y)\dd y\dd x=0,
\end{align*}
Using supp$(\phi)\subset[0,R-2]$ and hypothesis \textit{(CH2)}, we calculate  the following term:
\begin{align*}
&\left|\iint_{\mathbb{R}^{2}_{+}\setminus[0,R]^2}	\C_{\varepsilon}(x,y)f_{\varepsilon}(s,x)f_{\varepsilon}(s,y)\phi_{\varepsilon}(y)1_{[r_{\varepsilon}(x),R]}(y)\dd y\dd x\right|\\
&\qquad\leq2\bigg|\int_{R}^{\infty}\int_{0}^{R}\C_{\varepsilon}(x,y)f_{\varepsilon}(s,x)f_{\varepsilon}(s,y)\phi_{\varepsilon}(y)\dd y\dd x\bigg|
\\&\qquad\leq2||\phi||_{L^{\infty}}\int_{R}^{\infty}\int_{0}^{R}\C_{\varepsilon}(x,y)f_{\varepsilon}(s,x)f_{\varepsilon}(s,y)\dd y\dd x
\\&\qquad\leq2\|\phi\|_{L^\infty}\,\mathcal M\int_{R}^{\infty}\int_{0}^{R} f_{\varepsilon}(s,x)\, f_{\varepsilon}(s,y)\,\dd y\dd x
\\
&\qquad= 2\|\phi\|_{L^\infty}\,\mathcal M\Big(\int_{R}^{\infty} f_{\varepsilon}(s,x)\,\dd x\Big)\Big(\int_{0}^{R} f_{\varepsilon}(s,y)\,\dd y\Big)\\
&\qquad\leq 4\|\phi\|_{L^\infty}\,\mathcal M\cdot\frac{1}{R}||f^{\text{in}}||_{0,1}^2
\leq \frac{4}{R}\mathcal{M}||f^{\text{in}}||_{0,1}^2||\phi||_{L^{\infty}}.
\end{align*}

Therefore, it implies that above left side integral term is bounded. Similarly, we obtain the bound for the following term:
\begin{align*}
\left|\iint_{\mathbb{R}^{2}_{+}\setminus[0,R]^2}	\C(x,y)f(s,x)f(s,y)\phi(y)1_{[x,R]}(y)\dd y\dd x\right|\leq\frac{4}{R}\mathcal{M}||f^{\text{in}}||_{0,1}^2||\phi||_{L^{\infty}}.
\end{align*}
Hence, both tends to $0$ as $R\to+\infty$ uniformly with respect to $\varepsilon$.
Thus, combining above all results and by passing successive limits as $\varepsilon\to0$ and $R\to+\infty$, we obtain
\begin{align*}
&\int_{0}^{t}\int_{0}^{\infty}\int_{0}^{\infty}\K_{\varepsilon}(x,y)\;f_{\varepsilon}(s,x)f_{\varepsilon}(s,y)\left[yD_{\varepsilon}(\phi_{\varepsilon})(x)-\phi_{\varepsilon}(y)\right]1_{[0,r_{\varepsilon}(x)]}(y)\dd y\dd x\dd s\notag\\
&+\int_{0}^{t}\int_{0}^{\infty}\int_{0}^{\infty}\C_{\varepsilon}(x,y)\;f_{\varepsilon}(s,x)f_{\varepsilon}(s,y)\left[yD_{\varepsilon}(\phi_{\varepsilon})(x)-\phi_{\varepsilon}(y)\right]1_{[r_{\varepsilon}(x),+\infty]}(y)\dd y\dd x\dd s\notag\\
&\hspace{1.2cm}\longrightarrow\int_{0}^{t}\int_{0}^{\infty}\int_{0}^{\infty}\K(x,y)f(s,x)f(s,y)\left[y\partial_x\phi(x)-\phi(y)\right]1_{[0,x]}(y)\dd y\dd x\dd s\\\notag
&\hspace{2.2cm}+\int_{0}^{t}\int_{0}^{\infty}\int_{0}^{\infty}\C(x,y)f(s,x)f(s,y)\left[y\partial_x\phi(x)-\phi(y)\right]1_{[x,+\infty]}(y)\dd y\dd x\dd s,
\end{align*}
that is, 
\begin{align*}
&\notag\int_{0}^{t}\int_{0}^{\infty}\int_{0}^{r_{\varepsilon}(x)}\K_{\varepsilon}(x,y)\;f_{\varepsilon}(s,x)f_{\varepsilon}(s,y)\;[y\;D_{\varepsilon}(\phi_{\varepsilon})(x)-\phi_{\varepsilon}(y)]\;\dd y\dd x\dd s\\\notag
&+\int_{0}^{t}\int_{0}^{\infty}\int_{r_{\varepsilon}(x)}^{\infty}\C_{\varepsilon}(x,y)f_{\varepsilon}(s,x)f_{\varepsilon}(s,y)\left[yD_{\varepsilon}(\phi_{\varepsilon})(x)-\phi_{\varepsilon}(y)\right]\dd y\dd x\dd s\\\notag
&\hspace{1.2cm}\xrightarrow{\varepsilon\to0}\int_{0}^{t}\int_{0}^{\infty}\int_{0}^{x}\K(x,y)f(s,x)f(s,y)[y\partial_x\phi(x)-\phi(y)]\dd y\dd x\dd s\\
&\hspace{2cm}+\int_{0}^{t}\int_{0}^{\infty}\int_{x}^{\infty}\C(x,y)f(s,x)f(s,y)[y\partial_x\phi(x)-\phi(y)]\dd y\dd x\dd s.
\end{align*}
By Lemma \ref{lem_3_1}, Lemma \ref{lem_3_4} and Lemma \ref{lem_3_5}, we can easily see that for any $t>0$
\begin{align*}
\int_{0}^{\infty}f_{\varepsilon}(t,x)\phi_{\varepsilon}(x)\dd x\longrightarrow\int_{0}^{\infty}f(t,x)\phi(x)\dd x.
\end{align*}
Also, 
$\ds f_{\varepsilon}(0,x)\rightharpoonup f^{\text{in}}(x)$ in $L^1(\mathbb{R_+}).$
Lemma \ref{lem_3_5} along with previous convergence  result 
gives that
\begin{align}
\int_{0}^{\infty}f_{\varepsilon}(0,x)\phi_{\varepsilon}(x)\dd x\longrightarrow\int_{0}^{\infty}f^{\text{in}}(x)\phi(x)\dd x.
\end{align}
Gathering  all above results, we get that $f$ satisfies \eqref{2_5} and hence, it is a weak solution to the CCA equations \eqref{1_5}-\eqref{1_6}. 
\end{proof}


\section{Existence of solution in discrete regime}\label{sec_5}

Consider a weighted Banach space of real sequences $c=\left\{c_i\right\}_{i\geq1}$, defined by for all $r\geq0$,
\begin{align*}
	Y_r:=\left\{c=\left\{c_i\right\}_{i\geq1}:\sum_{i=1}^{\infty}i^r|c_i|<+\infty\right\},\hspace{0.17cm}\text{with norm}\hspace{0.17cm} ||c||_{Y_r}=\sum_{i=1}^{\infty}i^r|c_i|.
\end{align*}

Let $\ds Y_{r}^+:=\left\{c\in Y_r: c=\left\{c_i\right\}_{i\geq1}\geq0\hspace{0.17cm} \text{a.e.}\right\}$ denote a positive cone of  space $Y_{r}$ \cite{dubovskiui1994mathematical}.
Throughout this context we assume the following hypotheses:\\
\textit{Hypotheses: we consider the following hypotheses which are the discrete form of  hypotheses \textit{(CH1)} and \textit{(CH2)}: for all $i,j\geq1$,\\
$\it{(DH1a)}:$ $\K_{i,j}$ satisfies the growth condition, $\ds 
\lim_{j\to+\infty}\frac{\K_{i,j}}{j}=0, \quad \text{and}$\\
$\it{(DH2b)}:$  For some integer $m\geq1$,
$\C_{i,j}$ satisfies\; 
$\ds\sup_{j\geq m}\C_{i,j}\le \mathcal{M},\hspace{0.2cm}\text{where constant }\hspace{0.2cm}\mathcal{M}\geq1.$}

\begin{defn}[Weak solution]\label{def_4_1}
\textit{Let $T>0$. 
Assume that $\K_{i,j}$ and $\C_{i,j}$ are nonnegative and symmetric satisfying (DH1a) and (DH2b). Also, assume that $c^{\textnormal{in}}=\left\{c^{\textnormal{in}}_i\right\}_{i\geq1}$ be a sequence of nonnegative real numbers satisfying $c^{\textnormal{in}}\in Y_1^{+}$. A solution $c=\left\{c_i\right\}_{i\geq1}$ to the DCA equations \eqref{1_8}-\eqref{1_9} on $[0,T)$ is a sequence of nonnegative continuous functions such that, for all $i\geq1$ and each $t\in(0,T)$,
\begin{align*}
&(i)\quad c_i\in C([0,T)), \quad\sum_{j=i}^{\infty}\K_{i,j}c_ic_j\in L^1(0,t),\quad\text{and}\quad\sum_{j=i}^{\infty}j\C_{i,j}c_ic_j\in L^1(0,t),\\
& (ii)\quad c_i(t)=c^{\textnormal{in}}_i+\int_{0}^{t}\bigg[c_{i-1}(s)\sum_{j=1}^{i-1}j \K_{i-1,j}  c_{j}(s) - c_i(s) \sum_{j=1}^{i} j\K_{i,j} c_j(s)-\sum_{j=i}^{\infty} \K_{i,j} c_{i}(s)c_{j}(s)\\
&\hspace{3.67cm}+c_{i-1}(s) \sum_{j=i-1}^{\infty} j\C_{i-1,j}  c_{j}(s) - c_i(s) \sum_{j=i}^{\infty} j\C_{i,j} c_j(s)-\sum_{j=1}^{i}\C_{i,j} c_i(s)c_j(s)\bigg]\dd s.
\end{align*}}
\end{defn} 
The solution $c$ is global if $T=+\infty$.\\
To prove the global existence theorem, we truncate   the system \eqref{1_8}-\eqref{1_9}.
Let $N\geq3$ be  an integer.
Then, for $1\le i\le N $ and each $t\geq0$, we consider the following system of $N$ ordinary differential equations:
\begin{align}\label{4_2}
\frac{\dd c^N_i(t)}{\dd t} = Q^N_i(c^N(t)),\quad \text{with the initial data,}\quad c^N_i(0)=c^{\text{in},N}_i,
\end{align} 
where $c^N=\left\{c^N_i\right\}_{1\le i\le N}$, 
\begin{align}\label{4_3}
Q^N_i(c^N(t)):=&\notag c^N_{i-1}(t) \underbrace{\sum_{j=1}^{i-1}j \K_{i-1,j}  c^N_{j}(t)}_{\text{F}^1_i(t)} - c^N_i(t)\underbrace{\bigg[ \sum_{j=1}^{i} j\K_{i,j}c^N_j(t)+\sum_{j=i}^{N} \K_{i,j}c^N_j(t)\bigg]}_{\text{E}^1_i(t)}\\
&+c^N_{i-1}(t) \underbrace{\sum_{j=i-1}^{N} j\C_{i-1,j}c^N_{j}(t)}_{\text{F}^2_i(t)} - c^N_i(t)\underbrace{\bigg[ \sum_{j=i}^{N} j\C_{i,j} c^N_j(t)+\sum_{j=1}^{i} \C_{i,j}c^N_j(t)\bigg]}_{\text{E}^1_i(t)}.
\end{align}

Therefore,
\begin{align}\label{4_4}
\frac{\dd c^N_i(t)}{\dd t}+c^N_i(t)\bigg[E^1_i(t)+E^2_i(t)\bigg]=c^N_{i-1}(t)\bigg[F^1_i(t)+F^2_i(t)\bigg]. 
\end{align} 
\begin{Lemma}\label{lem_4_1}
Assume that kernels  $\K_{i,j}$ and $\C_{i,j}$ satisfy hypotheses \textit{(DH1a)} and \textit{(DH2b)}. Also, $c^{{\textnormal{in},N}}\in Y_1^{+}$. For every $N\geq3$, there exists a unique nonnegative solution $c^N=\left\{c^N_i\right\}_{1\le i\le N}$ in $C^1([0,+\infty),\mathbb{R}^N)$ to the system \eqref{4_2}-\eqref{4_3}. Furthermore, 
\begin{align}\label{4_5}
\sum_{i=1}^{N}ic^N_i(t)= \sum_{i=1}^{N}ic^{\textnormal{in},N}_i,\quad\text{for all}\hspace{0.2cm} t\in[0,+\infty).
\end{align}
\end{Lemma}
\begin{proof}
Assume that $c^{\text{in},N}=\left\{c^{\text{in},N}_i\right\}\in \mathbb{R}^N$. Notice that $Q^N$ is a locally Lipschitz continuous function. Therefore, the Cauchy-Lipschitz theorem guarantees the existence of a unique maximal solution  $c^N=\left\{c^N_i\right\}_{1\le i\le N}$ in $C^1([0, t^+(c^{\text{in},N});\mathbb{R}^N)$ to the system \eqref{4_2}-\eqref{4_3}, where either $t^+(c^{\text{in},N})=+\infty$ or,   $t^+(c^{\text{in},N})<+\infty$.

\noindent
\emph{Nonnegativity:}
For all $i\geq1$ and  each $t\in[0,+\infty)$ the solution $c_i^N$ of equation \eqref{4_4} holds  
\begin{align}\label{4_6}
c^N_i(t)=&\notag c^{\text{in},N}_i\exp\bigg(-\int_{0}^{t}\bigg[E^1_i(s)+E^2_i(s)\bigg]\dd s\bigg)\\
&+\int_{0}^{t}\exp\bigg(-\int_{s}^{t}\bigg[E^1_i(s_1)+E^2_i(s_1)\bigg]\dd s_1\bigg)c^N_{i-1}(s)\bigg[F^1_i(s)+F^2_i(s)\bigg]\dd s.
\end{align}
In equation \eqref{4_6}, all terms are nonnegative except the term
\begin{align}\label{4_7}
c^N_{i-1}(s)\bigg[F^1_i(s)+F^2_i(s)\bigg]=c^N_{i-1}(s)\bigg[\sum_{j=1}^{i-1} j\K^N_{i-1,j} c^N_{j}(s)+\sum_{j=i-1}^{N} j\C^N_{i-1,j}  c^N_{j}(s)\bigg].
\end{align}
Therefore, we have to check whether the term \eqref{4_7} is nonnegative or not.
Here, we propose the induction hypothesis for this.\\
$(i)$ For $i=1$, term \eqref{4_7} is zero which implies $c^N_1(t)\geq0$.\\
$(ii)$ For $i,k>1$ and $i\le k$,
it is true such that term \eqref{4_7} 
is nonnegative i.e. $c^N_i(t)\geq0$ for $i\le k$.\\
$(iii)$ For $i=k+1$, term \eqref{4_7} becomes
$$c^N_{k}(s)[F^1_{k+1}(s)+F^2_{k+1}(s)]=c^N_{k}(s)\bigg[\sum_{j=1}^{k} j\K^N_{k,j}  c^N_{j}(s)+\sum_{j=k}^{N} j\C^N_{k,j}  c^N_{j}(s)\bigg].$$
The above term is nonnegative due to $\K^N_{i,j}$ and $\C^N_{i,j}$ are nonnegative and  property $(ii)$.
Therefore, by induction hypothesis we obtain the nonnegativity of the solution.\\
For $1\le i,j\le N$ and each $c^N\in\mathbb{R^N}$,  the equation \eqref{1_10} gives
\begin{align*}
\sum_{i=1}^{N}iQ^N_i(c^N(t))=&\sum_{i=1}^{N}\sum_{j=1}^{i}j\K_{i,j} c^N_i(t) c^N_j(t)({i+1}-i) - \sum_{i=1}^{N}\sum_{j=i}^{N} i\K_{i,j}c^N_i(t)c^N_j(t)\\
&+\sum_{i=1}^{N}\sum_{j=i}^{N}j\C_{i,j}c^N_i(t) c^N_{j}(t)(i+1-i) -\sum_{i=1}^{N}\sum_{j=1}^{i} i\C_{i,j}c^N_i(t) c^N_{j}(t)=0
\end{align*}
Thereby,
\begin{align}
\frac{\dd}{\dd t}\sum_{i=1}^{N}ic^N_i(t)= 0, \hspace{0.4cm}\text{which implies}\hspace{0.4cm}
\sum_{i=1}^{N}ic^N_i(t)=\sum_{i=1}^{N}ic^{\text{in},N}_i.
\end{align}
Thus, the conservation of mass  for truncated problem is followed.
Moreover, due to aggregation, the total number of particles diminishes over time which deduces that for each $t\in[0,t^+(c^{\text{in},N}))$,
\begin{align*}
0\le c^N_i(t)\le \sum_{i=1}^{N}c^N_i(t)\le \sum_{i=1}^{N}c^{\text{in},N}_i\leq\sum_{i=1}^{N}ic_i^{\text{in},N}\leq \sum_{i=1}^{\infty}ic_i^{\text{in}}<+\infty .
\end{align*}
So,	 the above result implies that the solution $c^N$ to the system \eqref{4_2}-\eqref{4_3} does not blow up for   $t\in[0,\;t^+(c^{\text{in},N}))$ and hence,  $t^+(c^{\text{in},N})=+\infty$ for any $c^{\text{in},N}\in[0,+\infty)^N$.
\end{proof}
\begin{Lemma}\label{lem_4_2}
Assume that hypotheses  \textit{(DH1a)} and \textit{(DH2b)} 
hold and $c^N$ be a solution of the system \eqref{4_2}-\eqref{4_3} with the initial data $c^{{\textnormal{in},N}}\in Y_1^{+}$, 
for all $i\geq1$. Then, for each fixed $i\in\mathbb{N}$,  there exists  a constant $\Theta_i$ (depending on $i$ and initial mass 
 only) such that, for every $N\geq i$ and $t\in[0,+\infty)$,
\begin{align}\label{4_9}
\bigg|\frac{\dd c^N_i(t)}{\dd t}\bigg|\le \Theta_i.
\end{align}
\end{Lemma}
\begin{proof}
For every $i\geq1$ and using the hypotheses, we set
\begin{align}\label{4_10}
\mu_i:=\sup_{j}\frac{\K_{i,j}}{j}<+\infty,\quad\text{and}\quad\xi_i:=\sup_{j}\C_{i,j}<+\infty.
\end{align}
Therefore,   we  have 
\begin{align}\label{4_11}
\bigg|\frac{\dd c^N_i(t)}{\dd t}\bigg|\le&\bigg| c^N_{i-1}(t) \sum_{j=1}^{i-1}j \K_{i-1,j}  c^N_{j}(t)\bigg| +\bigg| c^N_i(t) \sum_{j=1}^{i} j\K_{i,j}c^N_j(t)\bigg|+\bigg|\sum_{j=i}^{N} \K_{i,j} c^N_i(t)c^N_j(t)\bigg|\notag\\
&+\bigg|c^N_{i-1}(t) \sum_{j=i-1}^{N} j\C_{i-1,j} c^N_{j}(t)\bigg| + \bigg|c^N_i(t) \sum_{j=i}^{N} j\C_{i,j}  c^N_j(t)\bigg|+\bigg|\sum_{j=1}^{i} \C_{i,j} c^N_i(t)c^N_j(t)\bigg|.
\end{align}
For $1\le j\le i-1$, we have $\ds j^2\leq j(i-1)$, whereas  for $1\le j\le i$,\; $\ds j^2\leq ij$. 
Hence,
\begin{align}\label{4_12}
\bigg|\frac{\dd c^N_i(t)}{\dd t}\bigg|\le&\notag (i-1) c^N_{i-1}(t) \mu_{i-1} \sum_{j=1}^{i-1}j\;c^N_{j}(t) +i c^N_i(t) \mu_i \sum_{j=1}^{i} jc^N_j(t)+\mu_i\sum_{j=i}^{N}j c^N_i(t)c^N_j(t)\\
&+\xi_{i-1}c^N_{i-1}(t) \sum_{j=i-1}^{N} jc^N_{j}(t) +\xi_ic^N_i(t)\sum_{j=i}^{N}j c^N_j(t)+\sum_{j=1}^{i}\C_{i,j} c^N_i(t)c^N_j(t).
\end{align}
Due to aggregation, we have $\ds \sum_{j=1}^{N}c^N_{j}(t)\leq\sum_{j=1}^{N}c^{\text{in},N}_{j}$.  
Also, each $c^N_i\leq\ds  \sum_{i=1}^{N}ic_i^{\text{in,N}}$. Using all previous estimates, the inequality \eqref{4_12} can be simplified as
\begin{align*}
\bigg|\frac{\dd c^N_i(t)}{\dd t}\bigg|\le&\notag \mu_{i-1}(i-1)c^N_{i-1}(t)\sum_{j=1}^{N}jc^N_{j}(t) + \mu_{i}i c^N_i(t) \sum_{j=1}^{N} jc^N_j(t)+\mu_ic^N_i(t)\sum_{j=1}^{N}jc^N_j(t)
\\&+\xi_{i-1}c^N_{i-1}(t) \sum_{j=1}^{N}j c^N_{j}(t) +\xi_{i}c^N_{i}(t) \sum_{j=1}^{N}j c^N_{j}(t)+\xi_i c^N_i(t)\sum_{j=1}^{i} c^N_j(t)\\
\le&\bigg(\mu_{i-1}+\xi_{i-1}+2(\mu_i+\xi_i)\bigg)\bigg(\sum_{j=1}^{\infty}jc^{\text{in}}_j\bigg)^2\\
=& \Theta_i\quad\text{where}\quad\Theta_i:=\bigg(\mu_{i-1}+\xi_{i-1}+2(\mu_i+\xi_i)\bigg)\bigg(\sum_{j=1}^{\infty}jc^{\text{in}}_j\bigg)^2.
\end{align*}	
\end{proof}
Our next goal is to prove the existence of the solution to the DCA equations \eqref{1_8}-\eqref{1_9} under some particular  conditions on $\K_{i,j}$ and $\C_{i,j}$ by the following proposition. 
\begin{Proposition}\label{prop_4_1}
Assume that $\left\{\K_{i,j}\right\}$ and $\left\{\C_{i,j}\right\}$  are nonnegative and symmetric satisfying  hypotheses \textit{(DH1a)} and \textit{(DH2b)}.
If $c^{\textnormal{in}}=\left\{c^{\textnormal{in}}_i\right\}_{i\geq1}$ be a sequence of nonnegative real numbers such that $c^{\textnormal{in}}\in Y_1^{+}$, 
then there exists at least a solution $c$ to the DCA equations  \eqref{1_8}-\eqref{1_9} on $[0,+\infty)$ such that mass conservation property  holds. 
\end{Proposition}
\begin{proof}
Let	$\left\{c^N_i\right\}_{N\geq 1}$ be the solution of the system of equations \eqref{4_2}-\eqref{4_3} obtained from Lemma \ref{lem_4_1}. Also, from Lemma \ref{lem_4_2},  $\left\{c^N_i\right\}_{N\geq 1}$ is uniformly bounded on $C[0,T)$ for $T\in(0,+\infty)$.
Thus,  \text{Arzel\`{a}-Ascoli} theorem  guarantees  that there exists a function $c=\left\{c_i\right\}_{i\geq1}$ and a subsequence of $\left\{c^N_i\right\}_{N\geq i}$ (not relabeled) such that 
\begin{align}\label{4_14}
c^N_i\longrightarrow c_i\quad\text{in}\quad  C([0,T)),\quad\text{for each}\quad i\geq1\hspace{0.15cm}\text{and}\hspace{0.15cm} T>0.
\end{align}
 Therefore,  $c_i$ is nonnegative function on $[0,+\infty)$ satisfying $\ds \sum_{i=1}^{\infty}ic_i(t)\le \sum_{i=1}^{\infty}ic^{\text{in}}_i$ for each $i\geq1$ and $t\geq0$. Therefore, for any  $t\geq0$ we obtain
\begin{align*}
\sum_{j=i}^{\infty}\K_{i,j}c_j(t)\le \bigg(\sup_{j\geq i}\frac{\K_{i,j}}{j}\bigg)\sum_{j=1}^{\infty}jc^{\text{in}}_j\le \mu_i\sum_{j=1}^{\infty}jc^{\text{in}}_j,
\end{align*}
where $\mu_i$ is defined in \eqref{4_10}.
Therefore, 
$\ds \sum_{j=i}^{\infty}\K_{i,j}c_j(t)\in L^1(0,t),\hspace{0.2cm}\text{for every}\hspace{0.2cm} t\geq0.$
By using growth condition \textit{(DH2b)}  and for each  $t\geq0$, we get 
\begin{align}
\sum_{j=i}^{\infty}j\C_{i,j}c_j(t)
\leq \mathcal{M}
\sum_{j=1}^{\infty}jc^{\text{in}}_j,
\end{align}
which implies that for all $i\geq1$,
\begin{align}
\sum_{j=i}^{\infty}j\C_{i,j}c_j(t)\in L^1(0,t), \quad t\geq0. 
\end{align}
Similarly, for all $i\geq1$ and $t\geq0$
\begin{align}
\sum_{j=i-1}^{\infty}j\C_{i-1,j}c_j(t)\in L^1(0,t).
\end{align}
Using convergence criteria of $c_i^N$ \eqref{4_14} on the product and sum of finite terms,
we can write that for each $t\geq0$ and $i\geq1$, 
\begin{align}
\int_{0}^{t}c^N_{i-1}(s) \sum_{j=1}^{i-1}j \K_{i-1,j}  c^N_{j}(s)\dd s\xrightarrow {N\to +\infty}\int_{0}^{t}c_{i-1}(s) \sum_{j=1}^{i-1} j\K_{i-1,j} c_{j}(s)\dd s.
\end{align}
Similarly, 
\begin{align}
\int_{0}^{t}c^N_{i}(s) \sum_{j=1}^{i}j \K_{i,j}  c^N_{j}(s)\dd s\xrightarrow {N\to+\infty}\int_{0}^{t}c_{i}(s) \sum_{j=1}^{i} j\K_{i,j} c_{j}(s)\dd s,\quad\text{and}
\end{align}
\begin{align}
\int_{0}^{t}c^N_{i}(s) \sum_{j=1}^{i} \C_{i,j} c^N_{j}(s)\dd s\xrightarrow{N\to+\infty} \int_{0}^{t}c_{i}(s) \sum_{j=1}^{i}\C_{i,j}c_{j}(s)\dd s.
\end{align}
For passing to the limit for other two terms, we fix $m\geq i$ and $N>m$. Therefore, for every $t\geq0$, we have
\begin{align}\label{4_21}
\notag\bigg|\int_{0}^{t}\bigg(\sum_{j=i}^{N}\K_{i,j} c^N_{i}(s) c^N_{j}(s)-\sum_{j=i}^{\infty}\K_{i,j} c_{i}(s) c_{j}(s)\bigg)\dd s\bigg|
\le&\bigg| \int_{0}^{t}\sum_{j=i}^{m} \K_{i,j}\bigg[c^N_{i}(s) c^N_{j}(s)-c_i(s)c_j(s)\bigg]\dd s\bigg|\notag\\
&+\bigg|\int_{0}^{t}\sum_{j=m}^{N}\K_{i,j} c^N_{i}(s)c^N_{j}(s)\dd s\bigg|+\bigg|\int_{0}^{t}\sum_{j=m}^{\infty}\K_{i,j} c_{i}(s)c_{j}(s)\dd s\bigg|.
\end{align}
Applying Lebesgue dominated convergence theorem and convergence criteria of $c_i^N$,  the first term of inequality 
\eqref{4_21} tends to zero as $N\to+\infty$. By  growth condition \textit{(DH1a)}, it follows that for each $t\geq0$,
\begin{align}
\int_{0}^{t}\sum_{j=m}^{N}\K_{i,j} c^N_{i}(s)c^N_{j}(s)\dd s\le \sup_{j\geq m}\frac{\K_{i,j}}{j}\int_{0}^{t}\sum_{j=m}^{N}jc^N_{i}(s)c^N_{j}(s)\dd s\le t\bigg(\sup_{j\geq m}\frac{\K_{i,j}}{j}\bigg)\bigg(\sum_{j=1}^{\infty}jc^{\text{in}}_{j}\bigg)^2,
\end{align}
and similarly, for each $t\geq0$,
\begin{align}
\int_{0}^{t}\sum_{j=m}^{\infty}\K_{i,j} c_{i}(s)c_{j}(s)\dd s\le t\bigg(\sup_{j\geq m}\frac{\K_{i,j}}{j}\bigg)\bigg(\sum_{j=1}^{\infty}jc^{\text{in}}_{j}\bigg)^2.
\end{align}
For $N\geq M$ and for every $t\geq0$, we may write
\begin{align}\label{4_24}
\notag	\bigg|\int_{0}^{t}\bigg(\sum_{j=i}^{N}j\C_{i,j} c^N_{i}(s) c^N_{j}(s)-\sum_{j=i}^{\infty}j\C_{i,j} c_{i}(s) c_{j}(s)\bigg)\dd s\bigg|
\le&\bigg| \int_{0}^{t}\sum_{j=i}^{m} j\C_{i,j}\left[c^N_{i}(s) c^N_{j}(s)-c_i(s)c_j(s)\right]\dd s\bigg|\notag\\
&	+\bigg|\int_{0}^{t}\sum_{j=m}^{N}j\C_{i,j} c^N_{i}(s)c^N_{j}(s)\dd s\bigg|+\bigg|\int_{0}^{t}\sum_{j=m}^{\infty}j\C_{i,j} c_{i}(s)c_{j}(s)\dd s\bigg|.
\end{align}
Using the similar argument as in \eqref{4_21} for above  inequality, the first term tends to zero as $N\to \infty$ and  by growth condition \textit{(DH2b)}, for each $t\geq0$, second term can reduced as
\begin{align}
\int_{0}^{t}\sum_{j=m}^{N}j\C_{i,j} c^N_{i}(s)c^N_{j}(s)\dd s\le \bigg(\sup_{j\geq m}\C_{i,j}\bigg)\int_{0}^{t}\sum_{j=m}^{N}{jc^N_{j}(s)}c^N_{i}(s)\dd s\le t\mathcal{M}\bigg(\sum_{j=1}^{\infty}jc^{\text{in}}_{j}\bigg)^2,
\end{align}
and similarly, for each $t\geq0$
\begin{align*}
\int_{0}^{t}\sum_{j=m}^{\infty}j\;\C_{i,j}c_{i}(s)c_{j}(s)\dd s\le t\mathcal{M}\bigg(\sum_{j=1}^{\infty}jc^{\text{in}}_{j}\bigg)^2.
\end{align*}
Hence, collecting above all results and passing to the successive limits as $N, m\to+\infty$, for each $t\geq0$ we get
\begin{align*}
\int_{0}^{t}\sum_{j=i}^{N}\K_{i,j}c^N_{i}(s) c^N_{j}(s)\dd s \longrightarrow\int_{0}^{t}\sum_{j=i}^{\infty}\K_{i,j} c_{i}(s) c_{j}(s)\dd s,\quad\text{and}
\end{align*}
\begin{align*}
\int_{0}^{t}\sum_{j=i}^{N}j\;\C_{i,j} c^N_{i}(s) c^N_{j}(s)\dd s \longrightarrow\int_{0}^{t}\sum_{j=i}^{\infty}j\C_{i,j} c_{i}(s) c_{j}(s)\dd s.
\end{align*}
Therefore,  $c=\left\{c_i\right\}_{i\geq1}$ satisfies equations  \eqref{1_8}-\eqref{1_9} and hence, $c$ is a solution to the DCA equations \eqref{1_8}-\eqref{1_9} in the sense of Definition \ref{def_4_1}.
\end{proof}

\section{Long-time behavior and occurrence of gelation}\label{sec_6}

\subsection{Propagation of Moments:}
 We consider a suitable space  such that initial moments belong to that space.
The $r$-th order moment  of the system at time $t\geq0$ is defined by $\ds M_r(t)=\sum_{i=1}^{\infty}i^rc_i(t)$, for $r\geq0$.
Here, zero-th moment $M_0(t)$ and first moment $M_1(t)$ respectively 
represent the total number and mass of the particles at any time $t$. 

\begin{Remark}
Assume that   $c^{\text{in}}\in Y_{r}^+$, for all $r\geq0$. The  $r$-th moment  is uniformly bounded in space  $Y_{r}^+$ for any finite time under specified kinetic-rates.
 In this regard, we  propose the following proposition for the growth conditions on $\K$ and $\C$ defined by
 \begin{align}\label{grt_1}
 	\K_{i,j}\leq \mathcal{A}_1ij,\quad \text{and}\quad 	\C_{i,j}\leq \mathcal{A}_2ij,\quad\text{for constants}\quad \mathcal{A}_1,\mathcal{A}_2>0.
 \end{align}
\end{Remark}

\begin{Proposition}
Assume that	$\K_{i,j}$ and $\C_{i,j}$  satisfy conditions \eqref{grt_1} and $\ds c^{\textnormal{in}}\in Y_{r}^+$, for $r\geq0$. Then,  the moments of all orders corresponding to  DCA equations \eqref{1_8}-\eqref{1_9}  is uniformly bounded in  $ Y_{r}^+$ over a finite time. 
\end{Proposition}

\begin{proof}
We have $M_r(0)\in Y_{r}^+$ for each $r\geq0$. Setting $\phi_i=1$ in moment equation \eqref{1_10}   and  using nonnegativity of $\K_{i,j}$, $\C_{i,j}$ and $c_i$, the time evolution of the zero-th  moment  is defined by 
\begin{align}\label{5_1}
\frac{\dd}{\dd t}\sum_{i=1}^{\infty}c_i(t)=-\sum_{i=1}^{\infty}\sum_{j=1}^{i}\K_{i,j}c_i(t)c_j(t)-\sum_{i=1}^{\infty}\sum_{j=i}^{\infty}\C_{i,j}c_i(t)c_j(t)\le 0.
\end{align}
$\text{This implies} \quad\ds \sum_{i=1}^{\infty}c_i(t)\le\sum_{i=1}^{\infty}c_i(0)=M_0(0)=\bar{M_0}$ (\text{say}).
Similarly, setting $\phi_i=i$ in  moment equation \eqref{1_10}, we obtain the time evolution of the first moment corresponding to the DCA equations \eqref{1_8}-\eqref{1_9}  is written as
\begin{align}\label{5_2}
\frac{\dd}{\dd t}\sum_{i=1}^{\infty}ic_i(t)=0\quad\text{which implies}\quad\sum_{i=1}^{\infty}ic_i(t)=\sum_{i=1}^{\infty}ic^{\text{in}}=M_1(0)=\bar{M_1},\quad\text{for}\hspace{0.2cm} t\geq0. 
\end{align}
Thus, the mass conservation law holds. 
Observe that  $j[(i+1)^2-i^2]-j^2=2ij+j-j^2\le(2ij+j)$ for all $i,j\geq1$.
Setting $\phi_i=i^2$ in moment equation  \eqref{1_10},  using growth conditions \eqref{grt_1} and mass conserving property \eqref{5_2}, the time evolution of the second moment is given by
\begin{align*}
	\frac{\dd M_2(t)}{\dd t}
	\le& \mathcal{A}_1\sum_{i=1}^{\infty}\sum_{j=1}^{\infty}(2i^2j^2+ij^2)c_i(t)c_j(t)+\mathcal{A}_2\sum_{i=1}^{\infty}\sum_{j=1}^{\infty}(2i^2j^2+j^2)c_i(t)c_j(t)\\
	\leq &
	\mathcal{A}\left[2M^2_2(t)+\bar{M}_1 M_2(t)\right]
\end{align*}
where, $\mathcal{A}=2\max\{\mathcal{A}_1,\mathcal{A}_2\}$. From above inequality, we get 
\begin{align}
	M_2(t)\leq \frac{\mathcal{A} \bar{M}_1 M_2(0) \exp(\mathcal{A} \bar{M}_1t)}{\mathcal{A} \bar{M}_1+2\mathcal{A}M_2(0)\left(1-\exp(\mathcal{A} \bar{M}_1t)\right)}
\end{align}
provided denominator is nonzero positive. Therefore, for finite time and given hypothesis,  the second moment $M_2(t)$ is bounded above by $
\bar{M}_2$.
Similarly, we can get the third order moment as
\begin{align*}
 M_3(t)
\leq \bar{M}_3.
\end{align*}
Continuing this process,  the $l$-th order moment  is given by 
$\ds M_l(t)\le \bar{M}_l$, for all $l\in\mathbb{N}$ where $\bar{M_l}$ is a constant. 
Next, To estimate $M_{l_1}(t)$ for a nonnegative real number $l_1$ such that  $l-1<l_1<l$ with a positive integer $l>1$.  Therefore,  previous results for $(l-1)$-th and $l$-th order moments give  $$\ds M_{l-1}(t)=\sum_{i=1}^{\infty}i^{l-1}c_i(t)\leq\bar{M}_{l-1}<+\infty  \quad\text{and} \quad M_{l}(t)=\sum_{i=1}^{\infty}i^{l}c_i(t)\leq\bar{M}_{l}<+\infty.$$
Let us write $l_1$ as $l_1=(l-1)\theta +(1-\theta)l$. Then, $\theta=l-l_1\in(0,1)$. Applying interpolation inequality for moments to the following sum: for each $i\geq1$ and $\theta\in(0,1)$,
\begin{align}
 M_{l_1}(t)=\sum_{i=1}^{\infty}i^{l_1}c_i(t)\leq \left(\sum_{i=1}^{\infty}i^{l-1}c_i(t)\right)^{\theta}\left(\sum_{i=1}^{\infty}i^{l}c_i(t)\right)^{1-\theta}\leq \bar{M}_{l-1}^{\theta}\bar{M}_l^{1-\theta}=\bar{M}_{l_1}, 
\end{align}
where $\bar{M}_{l_1}$ is a constant.
Hence, we conclude that the  moment  of all orders corresponding to the DCA equations \eqref{1_8}-\eqref{1_9} is uniformly bounded in finite time. Therefore,  $c_i(t)$ does not blow up  globally in finite time for $c_i^{\text{in}}\in Y_{r}^+$, for all $r\geq0$.
\end{proof}

\subsection{ Occurrence of gelation}
Under some  growth conditions on kinetic-kernels,  physical properties of particles such as size particle population, mass  in the system may evolve over a period of time. This evolution may lead the  occurrence of  gelation.  Therefore, it may be possible that the total mass of particles, $\ds M_1(t)$ may decrease to zero when time $t$ increases to infinity. In this case, we will inspect the long-time behavior  by the following proposition. 
\begin{Proposition}\label{prop_6_2}
Assume that  $c_i(t)$ be the nonnegative weak solution to the DCA equations  \eqref{1_8}-\eqref{1_9} on $[0,+\infty)$ such that the map  $t\mapsto ||c(t)||_{Y_1}$ is nonincreasing  on $[0,+\infty]$ and the condition $\ds \sum_{i=1}^{\infty}(1+i)c^{\textnormal{in}}_i<+\infty$ holds.
For $0\le t_1\le t_2$,  we have 
$\ds\sum_{i=1}^{\infty}c_i(t_2)\le \sum_{i=1}^{\infty}c_i(t_1).$
Also, assume  that for each $i,j\geq1$, there exists constants $K_1, K_2>0$ such that   $\ds\K_{i,j}\geq K_{1}ij$ and $\ds\C_{i,j}\geq K_{2}ij$. Then, 
\begin{align}\label{5_6}
\lim_{t\to +\infty}	||c(t)||_{Y_1}=0.
\end{align}
\end{Proposition}
\begin{proof}
Setting $\phi_i=1$ in  moment equation \eqref{1_10}, we obtain
\begin{align}\label{5_7}
\frac{\dd}{\dd t}\sum_{i=1}^{\infty}c_i(t)=-\sum_{i=1}^{\infty}\sum_{j=1}^{i}\K_{i,j}c_i(t)c_j(t)-\sum_{i=1}^{\infty}\sum_{j=i}^{\infty}\C_{i,j}c_i(t)c_j(t).
\end{align} 
Integrating  equation \eqref{5_7} with respect to $t$ over $[t_1,t_2]$, 
\begin{align}\label{5_8}
\sum_{i=1}^{\infty}c_i(t_2)-\sum_{i=1}^{\infty}c_i(t_1)=-\int_{t_1}^{t_2}\bigg(\sum_{i=1}^{\infty}\sum_{j=1}^{i}\K_{i,j}c_i(t)c_j(t)+\sum_{i=1}^{\infty}\sum_{j=i}^{\infty}\C_{i,j}c_i(t)c_j(t)\bigg)\dd t\leq0,
\end{align} 
due to the nonnegativity of $\K_{i,j}$ $\C_{i,j}$ and $c_i$. Thus $\ds \sum_{i=1}^{\infty}c_i(t)$ is nonincreasing.
 
For second assertion, let $m>1$ and $t\geq0$,  we set
\begin{align}
\mathcal{G}_m(t)=\sum_{i=1}^{m}c_i(t).
\end{align}

Set $\lambda_1\in (0,1)$, define  $$\phi^{\lambda_1}_i:=\min\left\{1,\frac{(m+\lambda_1-i)_+}{\lambda_1}\right\},\quad\text{for each}\quad i\geq1.$$ Note that $0\le \phi^{\lambda_1}_i\in L^{\infty}(\mathbb{R_+})$ and $\phi^{\lambda_1}_{i+1}-\phi^{\lambda_1}_i\le 0$.  Inserting  $\phi^{\lambda_1}_i$ in the moment equation \eqref{1_10}, integrating  with respect to $t$ over $[t_1,t_2]$ and using symmetry of kernels yield
\begin{align}\label{5_11}
\sum_{i=1}^{\infty}\left(c_i(t_2)-c_i(t_1)\right)\phi^{\lambda_1}_i&\notag\le -\int_{t_1}^{t_2}\bigg(\sum_{i=1}^{\infty}\sum_{j=1}^{i}\K_{i,j}c_i(t)c_j(t)\phi^{\lambda_1}_j+\sum_{i=1}^{\infty}\sum_{j=i}^{\infty}\C_{i,j}c_i(t)c_j(t)\phi^{\lambda_1}_j\bigg)\dd t\\
&\le -\frac{1}{2}\int_{t_1}^{t_2}\bigg(\sum_{i=1}^{m}\sum_{j=1}^{m}\K_{i,j}c_i(t)c_j(t)+\sum_{i=1}^{m}\sum_{j=m}^{\infty}\C_{i,j}c_i(t)c_j(t)\bigg)\dd t. 
\end{align}
Passing  $\lambda_1\to0$  in \eqref{5_11} gives
\begin{align}\label{5_12}
\mathcal{G}_m(t_2)-\mathcal{G}_{m}(t_1)\le -\frac{1}{2}\int_{t_1}^{t_2}\bigg(\sum_{i=1}^{m}\sum_{j=1}^{m}\K_{i,j}c_i(t)c_j(t)+\sum_{i=1}^{m}\sum_{j=m}^{\infty}\C_{i,j}c_i(t)c_j(t)\bigg)\dd t.
\end{align}
The above inequality indicates that $\mathcal{G}_m(\cdot)$ is  nonnegative and nonincreasing function of time. 
The  assumptions on  initial data  and growth conditions of  $\K_{i,j}$ and $\C_{i,j}$ ensure to pass the limit $m\to\infty$ in inequality \eqref{5_12}. Next, set $t_1=0$, $t_2=t$ and we simplify
\begin{align}\label{5_13}
\sum_{i=1}^{\infty}\left(c_i(t)-c_i(0)\right)	\leq -\frac{1}{2}\int_{0}^{t}\sum_{i=1}^{\infty}\sum_{j=1}^{\infty}\K_{i,j}c_i(s)c_j(s)\dd \leq -\frac{K_1}{2}\int_{0}^{t}\left(\sum_{i=1}^{\infty}ic_i(s)\right)^2\dd s.
\end{align} 
Due to time monotonicity of $\ds\sum_{i=1}^{\infty}c_i(t)$ defined by \eqref{5_8} and nonincreasing function  $t\mapsto ||c(t)||_{Y_1}$,   the above inequality \eqref{5_13} yields
\begin{align}
	\frac{K_1t}{2}\left(||c(t)||_{Y_1}\right)^2 \leq\sum_{i=1}^{\infty}c_i^{\text{in}}<+\infty,
\end{align}
which gives for $t\in(0,+\infty)$,
\begin{align}
||c(t)||_{Y_1}\leq \left(\frac{\ds2\sum_{i=1}^{\infty}c_i^{\text{in}}}{K_1t}\right)^{1/2}\quad\text{which implies}\quad \lim_{t\to+\infty}||c(t)||_{Y_1}=0.
\end{align}
  Hence, it completes the proof of proposition.
\end{proof}


\section{Numerical results}\label{sec_7}
This section presents the numerical study on the accuracy of solutions to the CA equations for two possible cases: (i) $\K=\C$ and (ii) $\K\neq\C$. In the first case, the mathematical equation reduces to a simplified form.
Moreover, in this simplified framework, the exact solutions are available, as reported in the work of Davidson \cite{MR3640930} for specific choices of kernels such as $\K=\C\equiv L$ (a constant) and $\K(x,y)=\C(x,y)=xy$.  Consequently, the case 
$\K=\C$ delivers  significant observation: it not only validates the accuracy of the numerical scheme but also demonstrates its importance to capture the theoretical behavior of the model under analytical growth conditions.
In contrast, the second case represents  a more general form that satisfies the motivation of the present work as it incorporates   additional nonlinear terms. The efficiency and accuracy of the numerical approximations in this case provides stronger validation of the robustness of the theoretical framework. First we analyze the graphs for $\K=\C$ and consequently proceed to solve two examples where $\K\neq\C$.  During numerical computations, we set $C=\lambda\K$, where $0\leq\lambda\leq1$ and both $\K$ and $\C$ satisfy the growth conditions of Theorem \ref{theo_3_1}.

The computational domain is chosen as $[0,10]$. For any $\varepsilon\in(0,1)$, the initial condition   $c^{\text{in},\varepsilon}=\ds \left\{c_i^{\text{in},\varepsilon}\right\}_{i\geq1}$ for the DCA equations is defined by \eqref{2_13}. Set $m=m(\varepsilon)=\ds\left[\frac{10}{\varepsilon}-\frac{1}{2}\right]$ which indicates the total number of cells $\Lambda_i^{\varepsilon}$ contained in the interval $[0,10]$.  The corresponding system of $m$ number of ordinary differential equations (ODEs) with the truncated kernels $\K^m$ and $\C^m$ is written as
\begin{align}\label{7_1}
	\frac{\dd c^m_i}{\dd t}=&\varepsilon \left (c^m_{i-1}\sum_{j=1}^{i-1}j\K^m_{i-1,j}c^m_j-c^m_{i}\sum_{j=1}^{i}j\K^m_{i,j}c^m_j-c^m_{i}\sum_{j=i}^{m}\K^m_{i,j}c^m_j\right)\notag\\
	&	+\varepsilon \left (c^m_{i-1}\sum_{j=i-1}^{m}j\C^m_{i-1,j}c^m_j-c^m_{i}\sum_{j=i}^{m}j\C^m_{i,j}c^m_j-c^m_{i}\sum_{j=1}^{i}\C^m_{i,j}c^m_j\right),\quad \text{with} \quad c^m_i(0)=c_i^{\text{in}},
\end{align}
 for all $i\in\left\{1,2,\dots,m\right\}$.  The system of ODEs  is solved numerically using Matlab's ODE45 code over a time interval $[0,t_{\text{max}}]$.
The approximated solution $f_{\varepsilon}$ is defined by 
\begin{align}\label{7_2}
	f_{\varepsilon}(t,x)=\sum_{i=1}^{m}c_i^m(t)1_{\Lambda_i^{\varepsilon}}(x),\quad\text{for all}\quad t\in [0,t_{\text{max}}]\quad\text{and}\quad x\in [0,10]. 
\end{align} 
 
 To assess further accuracy, we  compare approximated solutions $f_\varepsilon$ with exact solutions (wherever available in literature) for different values of $\varepsilon$. The  relative $L^1$ error is given by \cite{MR2158221}
 \begin{align}\label{7_3}
 	\mathcal{E}^1=\frac{||f_{\text{exact}}(t)-f_{\varepsilon}(t)||_{L^1}}{||f_{\text{exact}}(t)||_{L^1}}.
 \end{align}
 
 During numerical computations, dimensionless values for all the concerned quantities are considered. All simulations are performed in a HP Z6 G4 workstation and using MATLAB R2023b software.
\begin{example}\label{t_1}
	Consider kernels $\K=\C\equiv1$ with  initial condition $f(0,x)=x\exp(-x)$ for DCA equations \eqref{7_1}. The exact solution is collected from Davidson \cite{MR3640930} as
	\begin{align}\label{7_4}
		f(t,x)=\frac{(x-t)\exp(-x+t)}{1+t},\quad\text{for all}\quad x>t.
	\end{align}
\end{example}
  Numerical solutions for different values of $\varepsilon=0.05,0.01,0.005$ (as $\varepsilon\to0$) are plotted in figures \ref{fig_1a}  and \ref{fig_1b}  at times $t=1$ and $t=2.5$ respectively.  We can see a well propagation of the approximated solution to its exact solution as $\varepsilon\to0$. The  error graphs at different values of $\varepsilon$ are also plotted in figure \ref{fig_2} over time interval $[0, t_{\text{max}}]$. 
 \begin{figure}[H]\label{fig_1}
			\centering
			\subfloat[ Numerical and exact solutions at $t=1$
			\label{fig_1a}]{%
					\includegraphics[width=0.52\textwidth]{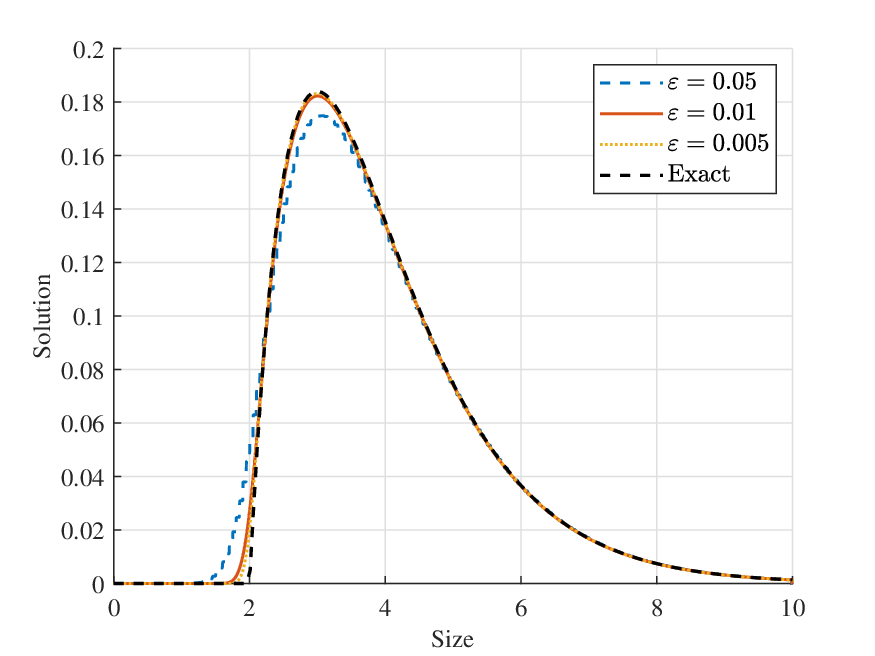}}
			\subfloat[Numerical and exact solutions at $t=2.5$ \label{fig_1b}]{%
					\includegraphics[width=0.52\textwidth]{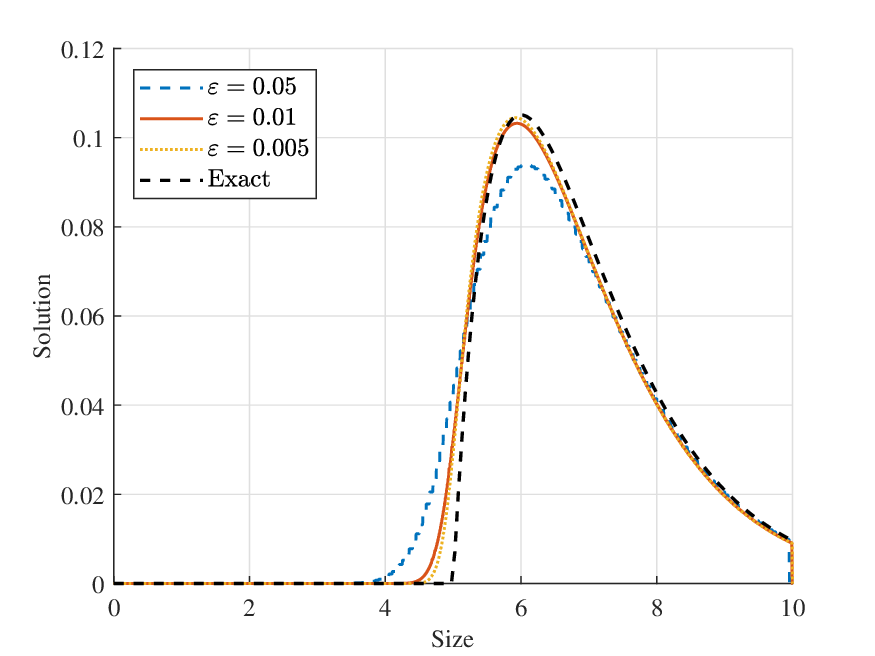}}
					\caption{Numerical and exact solutions at times $t=1$ and $t=2.5$ for test case \ref{t_1}}
		\end{figure}
		\begin{figure}[H]
			\centering
			\subfloat{%
				\includegraphics[width=0.524\textwidth]{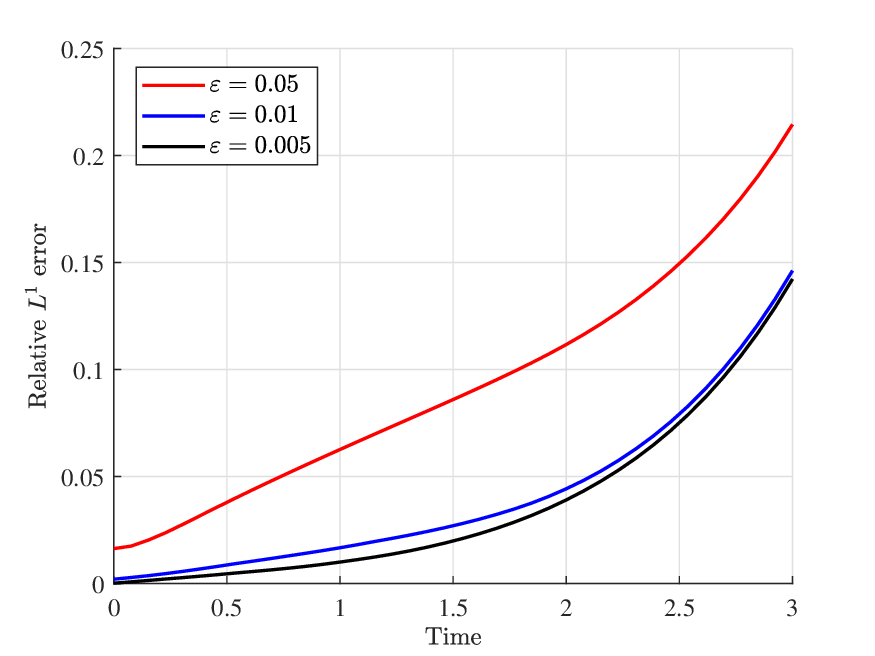}}
				\caption{Error with test case \ref{t_1}}
					\label{fig_2}
			\end{figure}
		\begin{example}\label{t_2}
Consider $\C=\lambda\K$ where $0\leq\lambda\leq1$ and $\K\equiv1$ with initial condition $f(0,x)=x\exp(-x)$ for the DCA equations \eqref{7_1}. Note that for $\lambda=1$, we get test case \ref{t_1} and for $\lambda\neq0$, we attain $\K\neq\C$ and both $\K$ and $\C$ satisfy growth conditions as in Theorem \ref{theo_3_1}. Also for $\lambda=0$, the model  reduces to pure OHS equation. 
\end{example}
 Figures \ref{fig_3a} and \ref{fig_3b} display the solutions obtained for different choices of kernels considered as $\K=1$ and hence, $$\ds \C=\ds \begin{cases}
 	0,&\mbox{for}\quad \lambda=0,\\
 	\frac{1}{2}, & \text{for}\quad \lambda=\frac{1}{2},\\
 	\frac{3}{4}, & \text{for}\quad \lambda=\frac{3}{4},\\
 	1, & \text{for}\quad \lambda=1,
 \end{cases}$$ and $\varepsilon=0.005$ together with exact solution \eqref{7_4}, plotted at times $t=1$ and $t=2.5$ respectively.
\begin{figure}[H]\label{fig_3}
	\centering
	\subfloat[  Numerical and exact solutions  at $t=1$
	\label{fig_3a}]{%
		\includegraphics[width=0.52\textwidth]{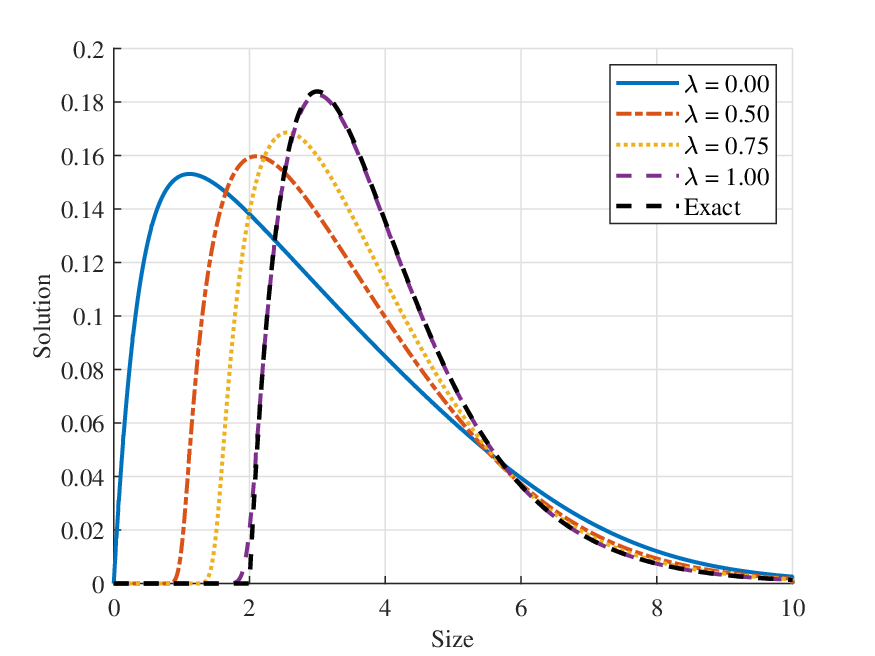}}
	\subfloat[Numerical and exact solutions   at $t=2.5$\label{fig_3b}]{%
		\includegraphics[width=0.52\textwidth]{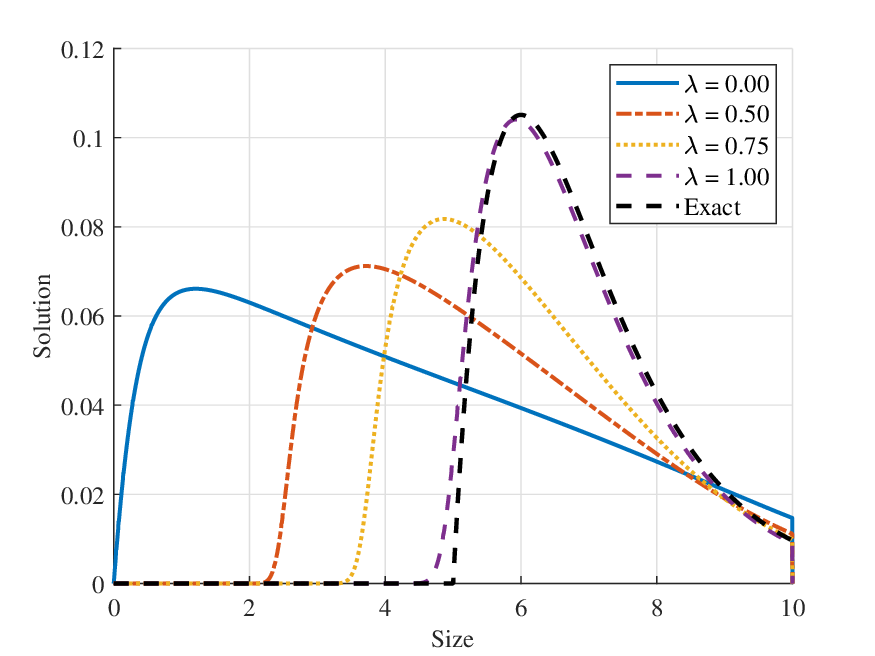}}
		\caption{Numerical and exact solutions for different values of $\lambda$ for test case \ref{t_2}}
	\end{figure}
	 
\begin{example}\label{t_3}
	 Finally, we consider 
		the pure OHS equation setting $\K\equiv1$ and  $\C\equiv 0$ with initial data $f(0,x)=\ds \frac{2}{M}1_{[0,M]}(x)$, for some $M>0$ for the DCA equations \eqref{7_1}. This example is taken from Bagland \cite{MR2158221} with  
		the exact solution 
		\begin{align}
			f(t,x)=\frac{2}{M(1+t)^2}1_{[0,M]}\left(\frac{x}{1+t}\right),\quad\text{for all}\quad (t,x)\in\mathbb{R}^2_{+},
		\end{align}
		to assess the validity of the proposed scheme in reduced kernel conditions.
\end{example}
Similar as test case \ref{t_1} we plot the numerical results for different values of $\varepsilon=0.05,0.01,0.005$ shown in figures \ref{fig_4a} and \ref{fig_4b} at times $t=1$ and $t=2.5$ respectively.
\begin{figure}[H]
	\centering
	\subfloat[ Numerical and exact solutions at $t=1$
	\label{fig_4a}]{%
		\includegraphics[width=0.52\textwidth]{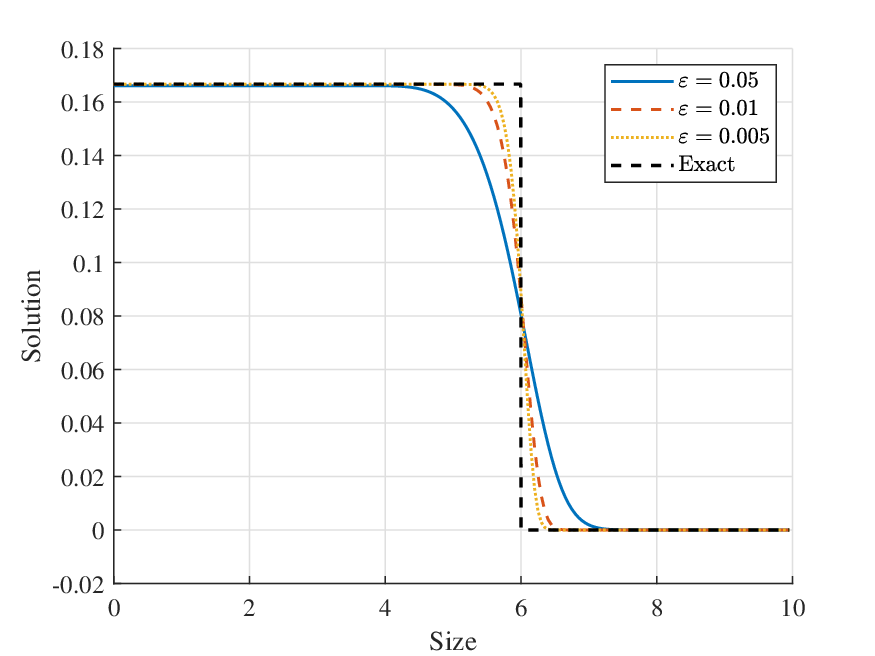}}
	\subfloat[Numerical and exact solutions at $t=2.5$ \label{fig_4b}]{%
		\includegraphics[width=0.52\textwidth]{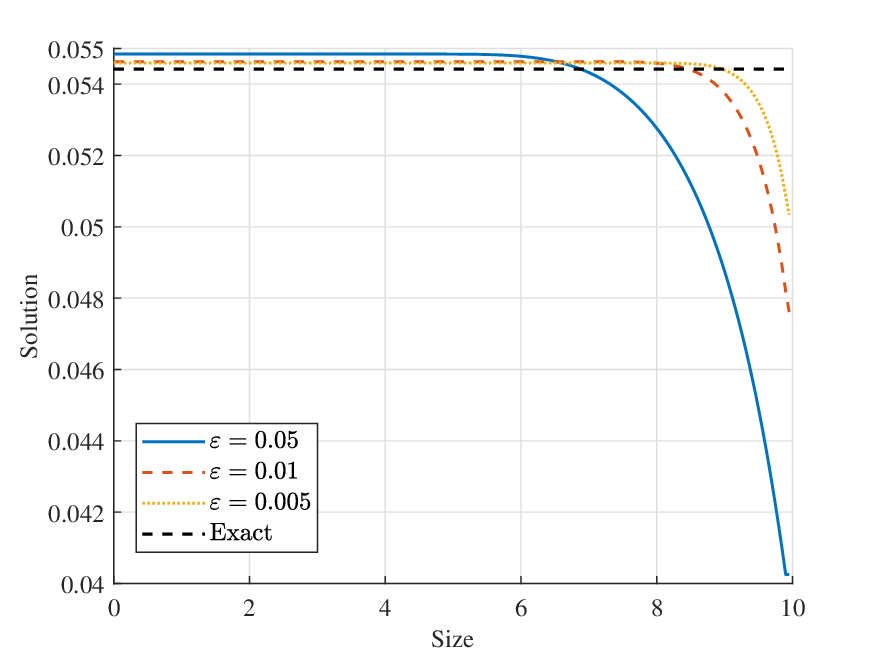}}
		\caption{Numerical and exact solutions at times $t=1$ and $t=2.5$ for $M=3$ for test case \ref{t_3}}\label{fig_4}
	\end{figure}
 
  The  errors  for test case \ref{t_3} are plotted in figure \ref{fig_5}  for $\varepsilon=0.05,0.01, 0.005$ over the time interval $[0,3]$. 
\begin{figure}[H]
	\centering
	\subfloat{%
		\includegraphics[width=0.525\textwidth]{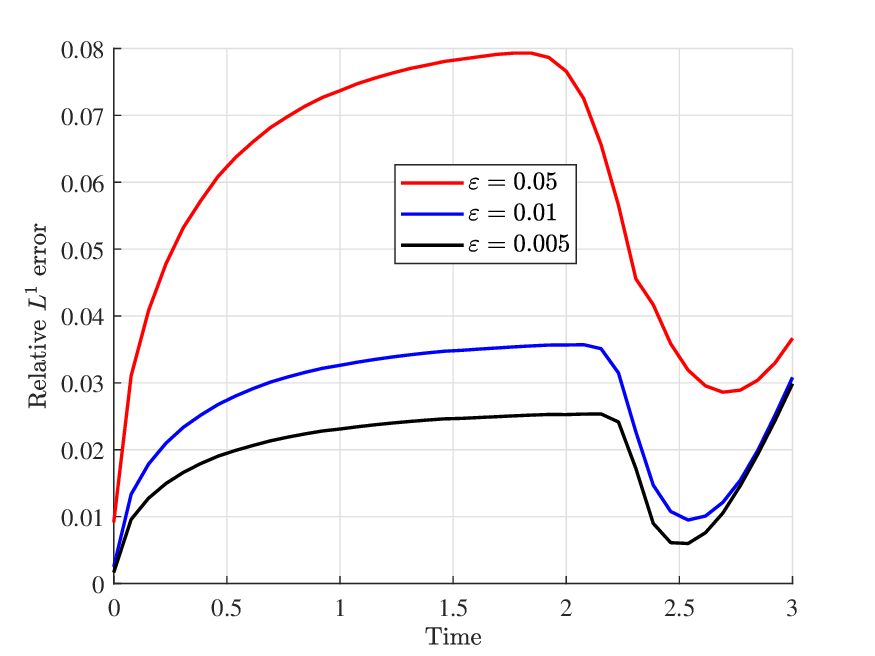}}
		\caption{Error with test case \ref{t_3} for $M=3$}\label{fig_5}
\end{figure}

\section{Conclusion}
This study is devoted to establish the connection between discrete and continuous of  CA equations by proving weak convergence. It is comprehensively shown that the sequence of approximated discrete equations converges to the solution of  continuous equation under weak compactness criterion. Furthermore, we proved the existence and uniqueness of the solution to the discrete version of  CA equation. Moreover, we investigated the long-time behavior for some particular kinetic-kernels along with propagation of moments over a finite time. Finally, some numerical computations successfully showed the agreement to the convergence of solutions to the DCA equations towards to the solution of CCA equation.

\section{Acknowledgment}
AG thanks University Grants Commission (UGC), Govt. of India for their funding support during her PhD program. 
JS thanks to Anusandhan National Research Foundation (ANRF), formerly  Science and Engineering Research Board (SERB), Govt. of India for their support through Core Research Grant (CRG / 2023 / 001483) during this work.

\section{Conflict of interest}	
All the authors certify that they do not have any conflict of interest.

\section{Data availability statement} The data that supports the findings of this study is available from the corresponding author upon reasonable request.

\bibliographystyle{unsrt}
\bibliography{bibliography}

\end{document}